\documentclass[a4paper, 11pt]{article}

\usepackage{mathrsfs,amssymb,amsmath, amsthm,color,tikz}

\usepackage[top=1.31in,bottom=1.29in,left=1.1in,right=0.9in]{geometry}
\usepackage[numbers]{natbib}
\usepackage[all]{xy}
\setcitestyle{open={},close={}}

\numberwithin{equation}{section}\theoremstyle{definition}
\swapnumbers

 \newtheorem{Theorem}[equation]{Theorem}
 \newtheorem{Prop}[equation]{Proposition}
 \newtheorem{Lemma}[equation]{Lemma}
 \newtheorem{Cor}[equation]{Corollary}
 
 \newtheorem{Defn}[equation]{Definition}
 \newtheorem{Example}[equation]{Example}
 \newtheorem{Remark}[equation]{Remark}

\newtheorem{Setup}[equation]{Set up}

 \makeatletter
\def\enumerate{\begingroup\ifnum\@enumdepth>3\@toodeep\else
      \advance\@enumdepth\@ne
      \edef\@enumctr{enum\romannumeral\the\@enumdepth}%
      \topsep\z@\parskip\z@
      \list{\csname label\@enumctr\endcsname}
        {\@nmbrlisttrue\let\@listctr\@enumctr
         \parsep\z@\itemsep\z@\topsep\z@
         \setcounter{\@enumctr}{0}
         \def \fMakelabel##1{\hss\llap{\rm ##1}}
       }\fi}

 \makeatother

\def\p{\mathfrak p}

\def\f{\mathfrak f}

\def\C{\mathbb C}
\def\N{\mathbb N}

\def\F{\mathbb F}

\def\cO{\mathcal O}
\def\cH{\mathcal H}

\def\cR{\mathcal R}

\def\cY{\mathcal Y}

\def\cN{\mathcal N}
\def\cC{\mathcal C}

\def\cF{\mathcal F}
\def\cS{\mathcal S}
\def\cM{\mathcal M}
\def\cD{\mathcal D}

\def\Pl{\mathbb{PL}}

\def\J{\makebox[5pt][l]{$^{J}$}}

\def\V{\hat{V}}

\DeclareMathOperator{\Hom}{Hom}

\DeclareMathOperator{\Ind}{Ind}

\DeclareMathOperator{\row}{row}

\DeclareMathOperator{\End}{End}

\DeclareMathOperator{\Pstab}{Pstab}

\DeclareMathOperator{\rank}{rank}

\DeclareMathOperator{\main}{main}

\DeclareMathOperator{\suppl}{suppl}

\DeclareMathOperator{\bi}{bi}
\DeclareMathOperator{\im}{Im}
\DeclareMathOperator{\supp}{supp}

\DeclareMathOperator{\Lie}{Lie}
\begin{document}

\title{On monomial linearisation and supercharacters  of \\ pattern subgroups}

\author{Qiong Guo$^{*}$, Richard Dipper$^{**}$\\ \\$^{*}${\footnotesize College of Sciences, Shanghai Institute of Technology} \\ {\footnotesize 201418 Shanghai, PR China}
\\ \scriptsize{E-mail:qiongguo@hotmail.com}\smallskip \\$^{**}$ {\footnotesize Institut f\"{u}r Algebra und Zahlentheorie}\\ {\footnotesize Universit\"{a}t Stuttgart, 70569 Stuttgart, Germany}
\\ \scriptsize{E-mail: richard.dipper@mathematik.uni-stuttgart.de}\\
\setcounter{footnote}{-1}\footnote{
{\scriptsize This work was  supported by NSFC no.11601338 and  the 
DFG priority programme SPP 1388 in representation theory, no. 99028426}}
\setcounter{footnote}{-1}
\setcounter{footnote}{-1}}
\date{}


 \maketitle

\begin{abstract}
Column closed pattern subgroups $U$ of the finite upper unitriangular groups $U_n(q)$ are defined as sets of matrices in $U_n(q)$ having zeros in a prescribed set of columns besides the diagonal ones. We explain Jedlitschky's construction of monomial linearisation [\cite{markus}] and apply this to $\C U$ yielding a generalisation of Yan's coadjoint cluster representations of [\cite{yan}] Then we give a complete classification of the resulting supercharacters, by  describing the resulting orbits and determining the Hom-spaces between  orbit modules.
\end{abstract}

\section{Introduction}

It is well known that finding the conjugacy classes of the finite unitriangular groups $U_n(q)$ of the upper unitriangular $n\times n$-matrices over some finite field $\F_q$ simultaneously for all natural numbers $n$ and all prime powers $q$ is a wild problem in the categorial meaning and hence in practice unsolvable. Hence most likely the same is true for the classification of the irreducible complex characters of the finite unitriangular groups. The supercharacter theory, as developed by Andr\'{e} and Yan in [\cite{andre}] respectively [\cite{yan}] for $U_n(q)$  and later generalised to $\F_q$-algebra groups and axiomatized by Diaconis and Isaacs in [\cite{super}] gives a (doable) approximation to the classification problems above. Supercharacters for a finite group $G$ are complex characters of $G$, such that every irreducible character is irreducible constituent of precisely one supercharacter. Supercharacters should be constant on superclasses, which are unions of conjugacy classes of $G$, and  every conjugacy class is contained in precisely one superclass. Moreover the number of superclasses and supercharacters should coincide. In the case  $G=U_n(q)$ the superclasses and supercharacters are classified by Andr\'{e} and Yan, and the supercharacter table is known. 

In this paper we are exclusively concerned with the supercharacter side of the theory. One basic feature of Yan's construction of supercharacters for $U=U_n(q)$ is finding a monomial basis for the group algebra $\C U$ such that the underlying permutation representation decomposes the monomial basis of $\C U$ into many orbits, decomposing the group algebra into  a corresponding direct sum of $\C U$-modules. In this paper we present a generalisation of Yan's original method, called ``monomial linearisation" due to Jedlitschky [\cite{markus}]. To make the paper self contained we present the proofs for Jedlitschky's construction here. 

Having found a monomial linearisation for $\C G$ of course raises the question  finding the $G$-orbits of the underlying permutation action of $G$ on the monomial basis of $\C G$.  This is different from the problem of classifying the resulting (monomial) orbit modules. Different orbits can lead to  characters which are not orthogonal, or even to isomorphic orbit modules and hence the inspection of Hom-spaces between orbit modules is a further problem. 

For so called pattern subgroups $U$ of $U_n(q)$ the monomial linearisation can be done by generalising Yan's construction. It can be shown then, that the characters afforded by the orbit modules are either equal or orthogonal. However, classifying the orbits and finding, which of the orbit modules are isomorphic and which afford different characters is an open problem for most pattern subgroups.

We deal here  with a special class of pattern subgroups $U_J$ of $U_n(q)$ given by prescribing some columns $1\le k_1<k_2<\ldots <k_m\leq n$ and taking all matrices  $u\in U_n(q)$ with $u_{ij}=0$ for all $1\le i <j$ where $ j\in \{k_1,k_2,\ldots ,k_m\}$. Here $u_{ij}$ denotes the entry at position $(i,j)$   of the matrix $u$. The set of these matrices is closed under multiplication and taking inverses and is a pattern subgroup of $U_n(q)$. We call them column closed subgroups. 

For $U_J$  we classify the orbits of the monomial basis of $\C U_J$. The orbit modules either afford orthogonal characters or are isomorphic and  the corresponding characters are the supercharacters of the supercharacter theory  defined in [\cite{super}]. We determine which of the orbit modules are isomorphic and which afford orthogonal characters.

\section{Monomial linearisation}\label{sec1}

In this section we present a general procedure to find in transitive permutation representations of a finite group $G$ a basis on which $G$ acts monomially. In the applications, which we have in mind, the action of $G$ on this monomial basis is not any more transitive, but decomposes the permutation representation into many orbit modules. This general method has been introduced by Jedlitschky in his thesis [\cite{markus}]. To make the article self contained, we present the construction including proofs.

Let $G$ be a finite group acting from the right on a finite abelian group $V$, written additively. For $g\in G, v\in V$, this action is denoted by $(v,g)\mapsto v.g$. Extending this by linearity to the group algebra $KV$ for any field $K$, the group $G$ act on $KV$ as group of $K$-algebra automorphisms from the right. In particular $G$ permutes the idempotents of $KV$, and hence, if $K$ is a splitting field for $V$ of characteristic not dividing the group order $|V|$ of $V$, we obtain a permutation action of $G$ on the basis of $KV$ consisting of the set of primitive idempotents of $KV$.

The group algebras $KG$  can be identified with the $K$-algebra $K^G$  of functions from $G$ to $K$, where the
multiplication of functions $\tau, \rho: G\rightarrow K $ is defined by
\[
(\tau\rho)(h)=\sum_{x\in G}\tau(x)\rho(x^{-1}h)
\]
The $K$-algebra isomorphisms $K^G\rightarrow KG$ is given by $\tau\mapsto \sum _{g\in G}\tau(g)g$ for $\tau\in K^G$.  Similarly we get $K^V\cong KV$.
The extension of the action of $G$ on $V$ to the group algebra $KV\cong K^V$ is given by $\tau\mapsto\tau.g$ for $\tau\in K^V, g\in G$, where $(\tau.g)(v)=\tau(v.g^{-1})$ for $v\in V$.

For simplicity let from now on $K$ be the complex field $\C$. Then $\C V$ is semi-simple. The irreducible $\C V$-modules are one dimensional and hence afford linear characters, that is homomorphisms from $V$ to $\C^*=\C\setminus\{0\}$, the multiplicative group of non-zero complex numbers. Let $\hat V=\Hom\big((V,+),\C^*\big)$ be the set of linear characters of $V$. Then $\hat V\subseteq \C^V\cong \C V$. Indeed for $\chi\in \hat V$ we have $\chi\leftrightarrow \sum_{v\in V}\chi(v)v$. But
$\sum_{v\in V}\chi(v)\in \C V$ is  up to a factor $|V|$ the primitive idempotent associated with the complex conjugate linear character $\bar{\chi}$ defined by $\bar{\chi}(v)=\big(\chi(v)\big)^{-1}=\chi(-v)\in \C^*$.
As a consequence $\hat V$ is a $\C$-basis of $\C^V$ and the linear extension of the $G$-action on $V$ permutes $\hat V$.

We next show that the number of orbits of $G$ on $V$ equals the number of orbits of $G$ on $\hat V$. In order to do this we employ the following lemma, well-known as the lemma that is not Burnside's:
\begin{Lemma}\label{burn}
Let the finite group $G$ act on a finite set $X$. For $g\in G$ let $X^g=\{x\in X\,|\,xg=x\}$. then the number $|X/G|$ of orbits of $G$ on $X$ is given as
\[
|X/G|=\frac{1}{|G|}\sum_{g\in G}|X^g|
\]
(i.e. as ``average number of points fixed by an element of $G$''). \hfill $\square$
\end{Lemma}

As a consequence we get our desired count of orbits of $G$ on $V$ and $\hat V$, respectively adapting an argument of Diaconis and Isaacs in [\cite{super}; 4.1] to this slightly different situation.

\begin{Cor}\label{2.2}
Let $G$ act on $V$ as above. Then the number of orbits of the actions of $G$ on $V$ and $\hat V$ coincide.
\end{Cor}
\begin{proof}
In view of \ref{burn} it suffices to show, that for all $g\in G$ the number of elements $v\in V$ satisfying $v.g=v$ and $\tau\in \hat V$ satisfying $\tau.g=\tau$ coincide.

Now for $g\in G$ and $\tau\in \hat V$ we have $\tau.g=\tau$ if and only if $\tau.g^{-1}=\tau$. But this is equivalent to 
\[
\tau(v.g)=\tau.g^{-1}(v)=\tau(v)
\]
and hence to $v.g-v\in \ker \tau=\{u\in V\,|\,\tau(u)=1\}$. For $g\in G$ let $W_g=\{v.g-v\,|\,v\in V\}$. Then $W_g\leq V$ and therefore $\{\tau\in \hat V\,|\,W_g\subseteq \ker \tau\}$ is in bijection with $\widehat{V/W_g}=\Hom\big((V/W_g,+), \C^*\big)$. Consequently the number of linear characters of $V$ fixed by $g\in G$ is exactly the index $[V:W_g]=\frac{|V|}{|W_g|}$ of $W_g$ in $V$.

Obviously the maps $v\mapsto v.g-v$ is an epimorphism from $V$ onto $W_g$ and the kernel $K$ of that map is the set of elements of $V$ fixed by $g\in G$. But $W_g\cong V/K$ and hence $|W_g|=\frac{|V|}{|K|}$. This implies $|K|=\frac{|V|}{|W_g|}=[V:W_g]$ and the result follows.
\end{proof}

Note that given a map $f: G\rightarrow V$ we always have a map $f^*:\C^V\rightarrow \C^G$ given by $\tau\mapsto \tau\circ f=f^*(\tau)$ for any map $\tau:V\rightarrow \C$. A map $f: G\rightarrow V$ is called a {\bf right 1-cocycle}, if it satisfies
\begin{equation}\label{2.3}
f(xg)=f(x).g+f(g),
\end{equation}
for all $x,g\in G$.

\begin{Theorem}[Jedlitschky, see {[\cite{markus}]}]\label{2.4}
Let $G$ act on $V$ by automorphisms from the right and let $f: G\rightarrow V$ be a right 1-cocycle. Then
the group algebra $\C^ V$ becomes a monomial $\C G$-module with monomial basis $\hat V=\Hom\big((V,+),\C^*\big)$ setting
\begin{equation}\label{monomial action}
\chi g=\chi(f(g^{-1}))\chi.g 
\end{equation}
for $\chi\in \hat V, g\in G$.  Here $\chi\mapsto \chi.g$ denotes the permutation action of $G$ on $\hat V$ obtained by extending the action of $G$ on $V$ to $\C^V\cong\C V$ by linearity. Moreover, with respect to the monomial action (\ref{monomial action}), the map $f^*: \C^V\rightarrow \C^G$ is  $\C G$-linear.
\end{Theorem}
\begin{proof}
In order to prove that \ref{monomial action} defines a $\C G$-module structure on $\C^V$ it suffices to check that $(\chi g) h=\chi(gh)$ for all $\chi\in \hat V, g,h\in G$ holds. In fact
\begin{eqnarray}
(\chi g) h&=&\big(\chi(f(g^{-1}))\chi.g \big)h=\chi(f(g^{-1}))\big((\chi.g)h\big)\nonumber\\
&=&\chi(f(g^{-1}))\big(\chi.g(f(h^{-1}))(\chi.g).h=\chi(f(g^{-1}))\big(\chi(f(h^{-1}).g^{-1})\chi.(gh)\nonumber\\
&=&\chi\big(f(g^{-1})+f(h^{-1}).g^{-1}\big)\chi.(gh)=\chi\big(f(h^{-1}g^{-1})\big)\chi.(gh)\nonumber\\
&=&\chi\big(f((gh)^{-1}))\chi.(gh)=\chi.(gh)\nonumber
\end{eqnarray}
as desired.

To prove that $f^*: \C^V\rightarrow \C^G$ is $\C G$-linear it suffices to show the maps $f^*(\chi g)$ and $f^*(\chi)g, (\chi\in \hat V, g\in G)$ from $G$ to $\C$ coincide . So let $h\in G$. Then
\begin{eqnarray}
\big(   f^*(\chi g)   \big) (h)&=&(\chi g) (f(h))\nonumber\\
&\stackrel{\ref{monomial action}}{=} &\chi\big( f(g^{-1}) \big)\chi.g\big(f(h)\big)=
\chi\big( f(g^{-1}) \big)\chi\big(f(h).g^{-1}\big)\nonumber\\
&=&\chi\big(   f(g^{-1})+     f(h).g^{-1}           \big)=\chi\big(   f(hg^{-1})  \big)\nonumber\\&=&\big(f^*(\chi)\big)(hg^{-1})=f^*(\chi)g(h)\nonumber
\end{eqnarray}
and hence $  f^*(\chi g) =f^*(\chi)g$ as desired.
\end{proof}

Note in the notation of \ref{2.4} that $\im f^*$ is a right ideal of $\C ^G\cong \C G$, and we call the monomial $\C G$-module $\C V\cong \C^V$ the \textbf{monomial linearisation} of $\im f^*$.

\begin{Lemma}\label{2.6}
Let $G, V$ be defined as in \ref{2.4} and let $f: G\rightarrow V$ a map. Then $f^*$ is injective, (surjective, bijective) if and only if $f$ is surjective (injective, bijective).
\end{Lemma}
\begin{proof}
Let $f$ be surjective and suppose $\tau\circ f=f^*(\tau)=f^*(\sigma)=\sigma\circ f$ for  $\tau, \sigma: V\rightarrow \C$. Since $f$ is surjective this implies $\tau=\sigma$, that is $f^*$ is injective. Now suppose that $f$ is not surjective and let $v\in V, v\notin \im f$. Let $\tau\in \C^V$. Define $\sigma: V\rightarrow \C$ by $\sigma(u)=\tau(u)$ for $v\not=u\in V$ and $\sigma(v)=\tau(v)-1$. Then $\tau\not=\sigma$ but $f^*(\tau)=\tau\circ f=\sigma\circ f=f^*(\sigma)$, that is $f^*$ is not injective. Thus $f^*$ is injective if and only if $f$ is surjective.

If $f$ is injective and $\rho: G\rightarrow \C$ is a map, then define $\tau: V\rightarrow \C$ by $\tau(v)=\rho(f^{-1}(v))$ for $v\in \im f$ and $\tau(v)=0$ for $v\notin \im f$. Then $f^*(\tau)=\rho$, and $f^*$ is surjective.
Finally let $g,h\in G, g\not=h$ and let $\sigma: G\rightarrow \C$ be given by $\sigma(g)=1$ and $\sigma(h)=\sigma(x)=0$ for $g\not=x\in G$. Suppose $f^*$ is surjective. Then there exists $\alpha: V\rightarrow \C$ with $\alpha\circ f=f^*(\alpha)=\sigma$. Thus 
$1=\sigma(g)=\alpha\circ f(g)\not=0=\sigma(h)=\alpha\circ f(h)$. Therefore $f(g)\not=f(h)$ proving that $f$ is injective.
\end{proof}

As an immediate consequence we have:

\begin{Theorem}
Let $f: G\rightarrow V$ be a bijective right 1-cocycle. Then the $\C G$-module $\C^V$ of \ref{monomial action} is isomorphic to the right regular representation $\C G_{\C G}$ of  $\C G$, the isomorphism from $\C^V$ to $\C^G\cong \C G$ given by $f^*$.
\end{Theorem}

Note that if $f:G\rightarrow V$ is bijective, then $f^{-1}: V\rightarrow G$ extended by linearity to $f^{-1}: \C V\rightarrow \C G$ is  $f^*: \C^V\rightarrow \C^G$      
 identifying  $\C ^V$ and $ \C V$ respectively $\C ^G$ and $ \C G$.

\begin{Remark}
In this paper we shall deal exclusively with bijective 1-cocycles and  monomial linearisation of the regular representation $\C G_{\C G}$ of $\C G$. More general it can be shown that the kernel $H=\ker f=\{g\in G\,|\, f(g)=0\}$ for any right 1-cocycle $f$ is an (in general not normal) subgroup of $G$ and that $\C^V$ is isomorphic by $f^*$ to the right ideal of $\C G$ generated by the trivial idempotent $|H|^{-1}\sum_{h\in H}h\in \C G$, if $f$ is surjective [\cite{embed}, Theorem 2.8].
Thus $\C^V$ is isomorphic to the transitive permutation module of $G$ acting on the cosets of $H$ in $G$ in this case. In the applications, we shall consider later with $H=\ker f=(1)$, i.e. $f$ is bijective, the action of $G$ on $V$ produces many orbits and hence in view of (\ref{2.2}) and (\ref{monomial action}) the regular $\C G$-module $\C G_{\C G}$ being isomorphic to the monomial $\C G$-module $\C^V$ decomposes into a direct sum of corresponding orbit modules arising from the monomial action of $G$ on $\hat V$.

\end{Remark}

So far we worked exclusively on the right using a right action of $G$ on $V$ and a right 1-cocycle $f: G\rightarrow V$. There is an obvious left hand sided analogue of all this, starting with a left action of $G$ on $V$ by automorphisms and a left 1-cocycle $f: G\rightarrow V$ satisfying $f(gx)=g.f(x)+f(g)$ for $g,x\in G$. Moreover, if $G$ acts on $V$ from the left as well as from the right by automorphisms such that $g.(v.h)=(g.v).h$ for all $g,h\in G, v\in V$ and if $f: G\rightarrow V$ is both, a left and right 1-cocycle, (\ref{monomial action}) and its left hand sided analogue define a monomial $\C G$-$\C G$-bimodule structure on $\C ^V$ with monomial basis $\hat V$. Moreover $f^*:\C ^V\rightarrow \C^G\cong \C G$ is a bimodule homomorphism. As a consequence $\im f^*$ is an ideal of $\C G$ and $\C^V\cong\, _{\C G}\C G_{\C G}$ as bimodule, if $f$ is bijective.

\section{Pattern subgroups and supercharacters}
Let $q$ be a power of some prime number $p$ and denote the field with $q$ elements by $\F_q$. Let $M_n(q)$ be the $\F_q$-algebra of $n\times n$-matrices with entries in $\F_q$ where $n\in \N$. For $A\in M_n(q), 1\leq i,j\leq n$, we denote the entry at position $(i,j)$ of $A$ by $A_{ij}$ and the matrix $A$ with $A_{ij}=1$ and $A_{kl}=0$ for $1\leq k,l\leq n$, $(k,l)\neq (i,j)$ by $e_{ij}$. Then $\{e_{ij}\,|\,1\leq i,j\leq n\}$ is an $\F_q$-basis of $M_n(q)$ and for $B\in M_n(q)$  we have $B=\sum_{1\leq i,j\leq n}B_{ij}e_{ij}$. If $M_n(q)^*=\Hom_{\F_q}(M_n(q), \F_q)$ denotes the dual $\F_q$-vector space, the $\F_q$-basis of $M_n(q)^*$ dual to the basis consisting of the matrix units $e_{ij}\, (1\leq i, j\leq n)$ is given by the coordinate functions $\epsilon_{ij}\,(1\leq i, j\leq n)$ defined by 
\[
\epsilon_{ij}(e_{kl})=\begin{cases}
1 & \text{for } (i,j)=(k,l),\\
0 & \text{otherwise}.
\end{cases}
\]
Then $\epsilon_{ij}(A)=A_{ij}\in \F_q$ for $A\in M_n(q),\, A=\sum_{1\leq i, j\leq n}A_{ij}e_{ij}$. We denote the set of all positions $\{(i,j),\,1\leq i\neq j\leq n\}$ of entries in $n\times n$-matrices by $\Phi$. Then $\Phi=\Phi^+\dot\cup \Phi^-$ with
$\Phi^+=\{(i,j)\,|\,1\leq i< j\leq n\}, \Phi^-=\{(i,j)\,|\,1\leq j< i\leq n\}$. For $(i,j)\in \Phi$ and $\alpha\in \F_q$ define $x_{ij}(\alpha)=E+\alpha e_{ij}$ where $E$ denotes the identity matrix $E=\sum_{i=1}^n e_{ii}$ and 
$X_{ij}=\{x_{ij}(\alpha)\,|\,\alpha\in \F_q\}$, (called \textbf{root subgroup}). Then
$X_{ij}$ is a subgroup of $GL_n(q)$ isomorphic to  $(\F_q,+).$ For $1\leq i\leq n$ and $ \alpha\in \F_q$ define $H_i=\{h_i(\alpha)=E+(\alpha-1) e_{ii}\,|\, \alpha \in \F_q^*=\F_q \setminus\{0\}\}$. By Bruhat decomposition we have $GL_n(q)=\langle X_{ij}, H_i\,|\, 1\leq i,j \leq n, i\neq j \rangle$.

For $A\in M_n(q)$ the set $\{1\leq i,j\leq n\,|\,A_{ij}\neq 0\}$  is called the \textbf {support} of $A$ and is denoted by $\supp(A)$. For any $L\subseteq \{(i,j)\,|\,1\leq i,j\leq n\}$ we may define $V_L=\{A\in M_n(q)\,|\,\supp A\subseteq L\}$. Then $V_L$ has basis $\{e_{ij}\,|\,(i,j)\in L\}$ and the dual space $V_L^*=\Hom_{\F_q}(V_L, \F_q)$ has basis $\epsilon_{ij}^L$ defined by
\[
\epsilon^L_{ij}(e_{kl})=\begin{cases}
1 & \text{for } (i,j)=(k,l)\\
0 & \text{otherwise}
\end{cases},\quad \text {where } (k,l)\in L.
\]
Note that if $J\subseteq L$ and $(i,j)\in J$, we get $\epsilon_{ij}^J$ as restriction of  $\epsilon_{ij}^L$ to $V_J$. Usually, if no ambiguities may arise, we hence omit superscripts $J$ and $L$. For $L$ as above and $A=\sum_{(i,j)\in L}\alpha_{ij}e_{ij}\in V_L$ we denote the corresponding linear form $\sum_{(i,j)\in L}\alpha_{ij}\epsilon_{ij}\in V_L^*$ by $A^*_L=A^*$. Thus the following lemma is obvious:

\begin{Lemma}
Let $J\subseteq \{(i,j)\,|\,1\leq i, j\leq n\}$ as above, $J\subseteq L$, and let 
$\pi_{_J}: M_n(\F_q)\rightarrow V_J: \sum_{(i,j)}\alpha_{ij}e_{ij}\mapsto\sum_{(i,j)\in J}\alpha_{ij}e_{ij}$ be the natural projection. Then the restriction  of $A^*_L$ to $V_J $  is $\big(\pi_{_J}(A)\big)^*_J.$ \hfill$\square$
\end{Lemma}

Let $1\leq k,l\leq n, k\not=l$ and $\alpha\in \F_q$. We want to investigate the action of $x_{kl}(\alpha)$ on $M_n(q)^*$. Obviously it suffices to determine $\epsilon_{ij}.x_{kl}(\alpha)$ for $1\leq i,j \leq n$. Expanding  $\epsilon_{ij}.x_{kl}(\alpha)=\sum_{1\leq s,t \leq n}\lambda_{st}\epsilon_{st}$ we may determine the coefficients in this expansion as 
$\lambda_{st}=\big(\epsilon_{ij}.x_{kl}(\alpha)\big)(e_{st})$. In fact
\begin{eqnarray*}
\lambda_{st}&=&\epsilon_{ij}\big((e_{st}x_{kl}^{-1}(\alpha)\big)=\epsilon_{ij}\big((e_{st}x_{kl}(-\alpha)\big)\\
&=& \epsilon_{ij}(e_{st}-\alpha \delta_{tk}e_{sl})=
\begin{cases}
\epsilon_{ij}(e_{st}), & \text{ for } t \neq k,\\
\epsilon_{ij}(e_{st}-\alpha e_{sl}), & \text{ for } t=k.
\end{cases}
\end{eqnarray*}
Note in the case $t=k$ that  $(i,j)=(s,t)=(s,k)$ implies $(i,j)\neq (s,l)$ and that $(i,j)=(s,l)$ implies $(i,j)\neq (s,t)=(s,k)$,
since $k\neq l$. Thus
\[
\lambda_{st}=
\begin{cases}
1 & \text{ if } (i,j)=(s,t)\\
-\alpha & \text{ if } (i,j)=(s,l) \text{ and } t=k\\
0 & \text{ otherwise.}
\end{cases}
\]
We have shown:
 \begin{Lemma}\label{3.2}
 Let $1\leq i,j, k,l\leq n, k\not=l$ and $\alpha\in \F_q$. Then
 \[
\epsilon_{ij}.x_{kl}=
 \begin{cases}
 \epsilon_{ij}& \text{ for } j\neq l\\
  \epsilon_{ij}-\alpha   \epsilon_{ik} & \text{ for } j= l.
 \end{cases}
 \]
 \hfill$\square$
 \end{Lemma}

Next we shall inspect the action of $h_k(\alpha)$ on $M_n(q)^*$, where  $1\leq k \leq n$ and $\alpha \in \F_q^*$. Writing $h_k(\alpha)=E+(\alpha -1)e_{kk}$ and  $\epsilon_{ij}.h_k(\alpha)=\sum_{1\leq s,t \leq n}\lambda_{st}\epsilon_{st}$ , we obtain:
\begin{eqnarray*}
\lambda_{st}&=&\big(\epsilon_{ij}.h_k(\alpha)\big)(e_{st})=\epsilon_{ij}\big(e_{st}h_k^{-1}(\alpha)\big)=\epsilon_{ij}\big(e_{st}h_k (\alpha^{-1})\big)\\
&=&\epsilon_{ij}\big(e_{st}(E+(\alpha^{-1}-1)e_{kk})\big)=\epsilon_{ij}\big(e_{st}+(\alpha^{-1}-1)\delta_{tk}e_{sk}\big)\\&=&
\begin{cases}
\epsilon_{ij}(e_{st}) & \text{ for } t\neq k\\
\epsilon_{ij}(\alpha^{-1}e_{st}) & \text{ for } t=k
\end{cases}\\&=&
\begin{cases}
1 & \text{ for } (i,j)=(s,t) \text{ and } t\neq k\\
\alpha^{-1 } & \text{ for } (i,j)=(s,t) \text{ and } t= k\\
0 & \text{ else.}
\end{cases}
\end{eqnarray*}
Therefore we have :
\begin{Lemma}\label{3.3}
Let $1\leq i,j,k\leq n$ and $\alpha\in \F_q^*$. Then
\[
\epsilon_{ij}.h_k(\alpha)=\begin{cases}
\epsilon_{ij} & \text{ for } j\neq k\\ \alpha^{-1 }   \epsilon_{ij}&   \text{ for } j= k
\end{cases}
\]
\hfill$\square$
\end{Lemma}

Combining \ref{3.2} and \ref{3.3}, we easily obtain the following corollary:
\begin{Cor}
Let $A\in M_n(q), g\in GL_n(q)$. Then
$
A^*.g=\big(A(g^{-t})\big)^*
$,
where $g^{-t}=(g^{-1})^t$ denotes the transposed matrix of the inverse matrix $g^{-1}$ of $g$. \hfill$\square$
\end{Cor}

Note that the corollary above can also be proved  directly by observing $A^*(B)=tr(A^tB)$ for $A,B\in M_n(q)$ and hence for $g\in GL_n(q)$: $
A^*.g(B)=A^*(B.g^{-1})=tr(A^tBg^{-1})=tr(g^{-1}A^tB)=(Ag^{-t})^*(B)
$. We present lemmas \ref{3.2} and \ref{3.3} since they explicitly describe the action of the generators of $GL_n(q)$ on the $\F_q$-basis of $M_n(q)^*$.

\medskip

A subset $J$ of $\Phi$ is called \textbf{closed}, if $(i,j), (j,k)\in J$ and $(i,k)\in \Phi$ implies $(i,k)\in J$. Note that if $J\subseteq \Phi^+\, (\Phi^-)$ the assumption $(i,k)\in \Phi$ above is satisfied automatically and hence $J\subseteq \Phi^+\,(\Phi^-)$ is closed if and only if $(i,j), (j,k)\in J$ implies $(i,k)\in J$.

Define $U=U_n(q), (U^-=U_n^-(q))$ to be the subgroups of  upper (lower) unitriangular $n\times n$-matrices. Then $U,\,(U^-)$  are $p$-Sylow subgroups of $GL_n(q)$. Moreover $V=\{u-E\,|\,u\in U\}$ is the Lie  algebra of $U$, where $E$ denotes the identity matrix $E=\sum_{i=1}^n e_{ii}$. 

It is well known that for a closed subset $J$ of $\Phi^+$   the set of matrices $A\in U^+$ with $\supp(A-E)\subseteq J$  is a subgroup of $U^+$  called  \textbf{pattern subgroup} and denoted by $U_J$. Moreover the associated Lie algebra $V_J=\Lie(U_J)$ is given as 
\[
V{_J}=\{A\in M_n(q)\,|\,\supp(A)\subseteq J\}=\{u-E\,|\,u\in U_J\}.
\] 

It is well known (see e.g. [\cite{carter}]) fact that for $J\subseteq \Phi^+$ closed we may fix an arbitrary linear ordering of $J$ and write every element $u$ of $U_J$ uniquely as
\begin{equation}\label{3.1}
u=\prod_{(i,j)\in J} X_{ij}(\alpha_{ij})
\end{equation}
for $\alpha_{ij}\in \F_q$, where the product is taken in the fixed given ordering. In particular $|U_J|=q^{|J|}$.

\begin{Remark}\label{3.6}
Note that multiplying $A\in M_n(q)$ from the right by $x_{ij}(\alpha)$ means adding $\alpha$ times column $i$ of $A$ to column $j$. Similarly multiplying $A$ from the left by $x_{ij}(\alpha)$ means adding $\alpha$ times row $j$ to row $i$ in $A$. So by (\ref{3.1}) (with $J=\Phi^+$) multiplying matrices in $M_n(q)$ from the right (left) by elements of $U$ can be performed by a sequence of elementary column (row)  operations from left to right (bottom to top respectively).
\end{Remark}

\begin{Cor}\label{3.7}
Let $J\subseteq \Phi^+$ closed. Then $U_J$ acts on $V_J$ by matrix multiplication from the left and the right. Moreover, for $A\in V_J, u\in U_J$ we have
\[
A^*.u=\big(\pi_{_J}(Au^{-t})\big)^*\in V^*,
\]
where $\pi_{_J}:M_n(q)\rightarrow V_J: B\mapsto \sum_{(i,j)\in J}B_{ij}e_{ij}$ is the natural projection. \hfill$\square$
\end{Cor}

In particular, for $(i,j)\in J\subseteq \Phi^+$ closed and $\alpha\in \F_q$ we have
\[
A^*.x_{ij}(\alpha)=\big(\pi_{_J}(A x_{ji}(-\alpha))\big)^*, \text{ for }A\in V_J.
\]
Recall from \ref{3.6} that $A  x_{ji}(-\alpha)$ adds $-\alpha$ times column $j$ of $A$ to column $i$. Applying $\pi_{_J}$ to the resulting matrix $B=Ax_{ij}(-\alpha)$ sets $B_{st}=0$ for $1\leq s,t \in n$ with $(s,t)\not\in J$. Thus for $J=\Phi^+$,
 $x_{ij}(\alpha)$ acts on $A^*$ as follows: 
  \begin{equation}\label{3.8}
A^*.x_{ij}(\alpha) = 
    \raisebox{-2cm}{
      \begin{tikzpicture}
        \fill[black,opacity=0.3] (0.8,4)--(1.2,4)--(1.2,1.2)--(0.8,1.2)--cycle;
        \fill[white] (0.8,3.2)--(1.2,3.2)--(1.2,2.8)--(0.8,2.8)--cycle;
        \fill[black,opacity=0.4] (0.8,2.8)--(1.2,2.8)--(1.2,1.2)--(0.8,1.2)--cycle;
        \draw[black,thick] (0.8,4)--(1.2,4)--(1.2,1.2)--(0.8,1.2)--cycle;
        \fill[black,opacity=0.3] (2.8,4)--(3.2,4)--(3.2,0.8)--(2.8,1.2)--cycle;
        \draw[black,thick] (2.8,4)--(3.2,4)--(3.2,0.8)--(2.8,1.2)--cycle;
        \draw[black](1,2)--(1.5,1) node[below] {set to zero};
        \draw[black](3.7,3)--(4.4,2.7) node[right] {$A^*$};
        \draw (4,0)--(4,4)--(0,4)--cycle;
        %
        \draw(1,3) node{$i$};
        \draw(3,1) node{$j$};
        \draw[->](3,4.1)--(3,4.3)--(1,4.3)--(1,4.1);
        \draw (2,4.5) node {$-\alpha$ times};
      \end{tikzpicture}
    }
  \end{equation}
 We indicate here the dual $A^*$ of the upper-triangle nilpotent matrix $A$ as triangle omitting superfluous zeros. If $J$ is a closed subset in $  \Phi^+$, $(i,j)\in J$ for the action of $X_{ij}$ on $A^*$ all  positions not in $J$ should be set to zero in (\ref{3.8}) as well. 

The application of $\pi_{_J}$ corresponds to restricting $\epsilon_{ij}\in M_n(q)^*$ to $\epsilon_{ij}^J\in V_J^*$ for $(i,j)\in J$. In previous papers we called this ``truncation'', but now, following Yan [\cite{yan}], we call it restriction and the action described in illustration \ref{3.8} ``restricted column operation''. 

The left action of $GL_n(q)$ on $M_n(q)^*$ derived from left multiplication similarly can be stated as:
\begin{equation}\label{3.9}
g.A^*=(g^{-t}A)^*
\end{equation}
and hence for $J\subseteq \Phi^+, (i,j)\in J, \alpha\in \F_q$ we can describe 
$x_{ij}(\alpha).A^*=\big(\pi_{_J}(x_{ji}(-\alpha)A)\big)^*$, that is by the restricted row operation adding $-\alpha$ times row $i$ onto row $j$ and project the resulting matrix $B$ to $V_J$ to obtain $x_{ij}(\alpha).A^*=B^*\in V_J^*$:
\begin{equation}\label{3.10}
x_{ij}(\alpha).A^* = 
    \raisebox{-2cm}{
      \begin{tikzpicture}
        \fill[opacity=0.3] (4,0.8)--(4,1.2)--(1.2,1.2)--(1.2,0.8)--cycle;
        \fill[white] (3.2,0.8)--(3.2,1.2)--(2.8,1.2)--(2.8,0.8)--cycle;
        \fill[black,opacity=0.4] (2.8,0.8)--(2.8,1.2)--(1.2,1.2)--(1.2,0.8)--cycle;
        \draw[thick]  (4,0.8)--(4,1.2)--(1.2,1.2)--(1.2,0.8)--cycle;
        %
        \draw[thick] (4,2.8)--(4,3.2)--(0.8,3.2)--(1.2,2.8)--cycle;
        \fill[opacity=0.3]  (4,2.8)--(4,3.2)--(0.8,3.2)--(1.2,2.8)--cycle;
        \draw[black](1.5,0.9)--(1.7,0.3) node[below] {set to zero};
        \draw[black](3.7,3.6)--(4.4,4.1) node[right] {$A^*$};
        \draw (4,0)--(4,4)--(0,4)--cycle;
        %
        \draw(1,3) node{$i$};
        \draw(3,1) node{$j$};
        \draw[->](4.1,3)--(4.3,3)--(4.3,1)--(4.1,1);
        \draw(5.5,2) node {$-\alpha$ times};
      \end{tikzpicture}
    }
\end{equation}
Positions in the triangle not in $J$ are set to be zero as well.

Let $J\subseteq \Phi^+$ be closed.  Since $\{\epsilon_{ij}\,|\,(i,j)\in J\}$ is a basis of $V_J^*$, we have $V_J^*=\{A^*\,|\, A\in V_J\}$. Now $U_J$ acts on $V_J^*$ from left and right by multiplication (denoted here as ``.'') and hence on $V_J^*$ as well. By  [\cite{super}, Lemma 4.1] the numbers of $U_J$-orbits (left, right, bi) on $V_J$ and $V_J^*$ coincide.

\begin{Remark}\label{lide}
We choose once for all a non trivial character $\theta:(\F_q,+)\longrightarrow \C^{*}$. Obviously, for  $A\in V_J$, $\theta$ composed with the $\F_q$-linear map $A^*: V_J\rightarrow \F_q$ is a linear character of $V_J$, denoted by $[A]$. Obviously 
$\hat V_J=\{[A]\,|\,A\in V_J\}$. Note that under the identification $\C^{V_J} \rightarrow \C(V_J,+):\tau\mapsto \sum_{v\in V_J} \tau(v)v$ the linear character $[A]\in \hat{V}_J$ for $A\in V_J$ is mapped to $q^{|J|}$ times the idempotent of $\C V_J$, associated with the complex conjugate character $\bar{[A]}$ which is obviously $[-A]$, since $V_J$ is written additively. In [\cite{embed}] we called $e_{_A} = q^{-|J|}[-A]\in \C V_J$ hence \textbf{lidempotent}, to distinguish those from idempotents of the group algebra $\C U_J$ and since it is really an idempotent for the additive group of the Lie algebra $\text{Lie}(U_J) = V_J$. Moreover, for simplicity 
 we call the linear characters $[A]$  in $\V$  ``lidempotents" as well, although $[A]$  is the multiple $q^{|J|}e_{_{-A}}$ of the idempotent $e_{_{-A}}$ in the group algebra  $\C^V  = \C V$.
\end{Remark}

\begin{Prop}\label{3.11}
Let $J\subseteq \Phi^+$ be closed, $A\in V_J, u\in U_J$. 
Then extending as in section 2 the action of $U_J$ on $V_J$ to $\C V_J$,  this action from the left and the right  satisfies $[A].u=\theta\circ(A^*.u)$ and $u.[A]=\theta\circ(u.A^*)$ and hence 
$[A].u=[\pi_{_J}(Au^{-t})]$ and $u.[A]=[\pi_{_J}(u^{-t}A)]$. 
\end{Prop}
\begin{proof}Let $B \in V_J$. Then
$
([A].u)(B)=([A])(Bu^{-1})=(\theta\circ A^*)(Bu^{-1})=\theta\big(A^*(Bu^{-1})\big)=\theta\big((A^*.u)(B)\big).
$
By \ref{3.7} we have $A^*.u=\big(\pi_{_J}(Au^{-t})\big)^*$. Hence we obtain: 
\[
[A].u=\theta\circ(A^*.u)=\theta\circ\big(\pi_{_J}(Au^{-t})\big)^*=[\pi_{_J}(Au^{-t})].
\]
Similarly we get 
\[
u.[A]=\theta\circ(u.A^*)=\theta\circ\big(\pi_{_J}(u^{-t}A)\big)^*=[\pi_{_J}(u^{-t}A)].
\]
\end{proof}

\begin{Lemma}\label{3.12}Let $J\subseteq \Phi^+$ be  closed. Then the map $f: U_J\rightarrow V_J: u\mapsto u-E$ where $u\in U_J$, is a left, right and bijective 1-cocycle. In particular $f^*: \C^{V_J}\rightarrow \C^{U_J}$ is bijective.
\end{Lemma}
\begin{proof}
Let $x,g \in U_J$. Then $f(xg)=xg-E=(x-E)g+(g-E)=f(x)g+f(g)$ and  $f(gx)=gx-E=g(x-E)+(g-E)=gf(x)+f(g)$ , hence $f$ is a left and right 1-cocycle. In addition, it is obviously bijective and hence $f^*$ is also bijective by Lemma \ref{2.6}.
\end{proof}

Using  Jedlitschky's theorem \ref{2.4} 
 we derive an action of $\C U_J$ on $\C V_J$ such that $U_J$ acts on $\hat V_J$ monomially. More precisely,
\begin{Cor}\label{3.13}
 For $(i,j)\in J\subseteq \Phi^+$ and $A^*\in \hat V_J$, we have
\begin{itemize}
\item [(1)] $[A] x_{ij}(\alpha) =\theta(-\alpha A_{ij})[B]$ where $\alpha\in \F_q^*$ and $B$ is obtained from $A$ by adding $-\alpha$ times column $j$ to column $i$ and then setting all the positions not in $J$ back to zero. We call this ``{\bf restricted column operation}'' from left to right.
\item [(2)]  $ x_{ij}(\beta) [A]=\theta(-\beta A_{ij})[B]$ where $\beta\in \F_q^*$ and $B$ is obtained from $A$ by adding $-\alpha$ times row $i$ to row $j$ and then setting all the positions not in $J$ back to zero. We call this ``{\bf restricted row operation}'' from bottom up.
\end{itemize}
\end{Cor}

\begin{Defn}\label{orbits}
Let $J\subseteq \Phi^+$ be closed. For $A\in V_J$, we set $$\cO_A^l=U_{J}.[A]\,, \quad  \cO_A^r=[A].U_{J}\,,\quad \text { and } \quad \cO_A^{\bi}=U_{J}.[A].U_{J}.$$ Thus these are $U_J$-orbits given by $[A] $ under the permutation action of $U_J$ on $\hat V_J$.
However when we consider the $\C$-span of these orbits, denoted by  $\C \cO_A^r, \C \cO_A^l $ and $\C \cO_A^{\bi}$ respectively, we shall,  if not stated otherwise, consider them as monomial $\C U_J$-modules given in \ref{2.4} and  \ref{3.13}.
\end{Defn}

\begin{Remark}\label{HomSpace}
Note that $f^*(\C \cO_A^l)=\C U_J f^*([A])$ is a left ideal of $\C U_J$,  $f^*(\C \cO_A^r)= f^*([A])\C U_J$ is a  right ideal of $\C U_J$ and $f^*(\C \cO_A^{\bi})=\C U_J f^*([A])\C U_J$ is an ideal of $\C U_J$. For $ A, B\in V_J$, we have:
\begin{eqnarray*}
\Hom_{\C U_J}\big(\C \cO_A^r,\C \cO_B^r\big) & \cong & \Hom_{\C U_J}\big(f^*([A])\C U_J, f^*([B])\C U_J\big) \\ 
 & \cong & \C U_J f^*([A]) \cap f^*([B])\C U_J\\
 & \cong\ & \C  \cO_A^l \cap \C\cO_B^r \\
 & \cong\ & \C  (\cO_A^l \cap \cO_B^r)\\
& \cong &\Hom_{\C U_J}\big(\C \cO_B^l,\C \cO_A^l\big) , 
\end{eqnarray*}  
since $\C U_J$ is a self-injective algebra.
More precisely, if $[C]\in  \cO_A^l \cap \cO_B^r  $ then there exist $\lambda,\mu \in\C^*$ and  $x,y\in U_J$ such  that $[C]=\lambda x[A] = \mu [B]y$. Then the $\C U_J$-homomorphism from $\C \cO_A^r$ to $\C \cO_B^r$ corresponding to $[C]$ is given by left multiplication by $\lambda x$. Since  $\cO_C^r = \cO_B^r$,  this is an epimorphism, and it is injective, since $\lambda x$ is invertible. Similarly right multiplication by $\mu y$ is a $\C U_J$-homomorphism depending on $[C]$ alone from $\C \cO_B^l$ to $\C \cO_A^l$. In particular we see that $\End_{\C U_J}(\C  \cO_A^r) \cong  \End_{\C U_J}(\C  \cO_A^l)$ as algebras.
\end{Remark}

For an arbitrary finite group $G$ we say that two $\C G$-modules are \textbf{disjoint}, if they have no irreducible constituent in common, or equivalently, if their characters are orthogonal.

We have shown:

\begin{Lemma}\label{equalororth}
Let $J\subseteq \Phi^+$ be closed and $A, B\in V_J$. Then either $\C \cO_A^r\cong \C \cO_B^r$ and $\C \cO_A^l\cong \C \cO_B^l$ or  $\C \cO_A^r, \C \cO_A^r$ (respectively $\C \cO_A^l, \C \cO_A^l$) are disjoint, that is the characters afforded by $\C \cO_A^r$ and $ \C \cO_A^r$ (respectively by $\C \cO_A^l$ and $ \C \cO_A^l$) are orthogonal.
\end{Lemma}

This implies in particular, that all orbit modules contained in a biorbit are isomorphic and orbit modules contained in different bimodules are disjoint. Since in addition the sum of all biorbit modules is the regular $\C U_J$-bimodule, we have:
\begin{Cor}\label{leftright}
Let $J\subseteq \Phi^+$ be closed and $A, B\in V_J$. Then 
$$\C \cO_A^r\cong \C \cO_B^r\Longleftrightarrow [B]\in \cO_A^{\bi}\Longleftrightarrow\C \cO_A^l\cong \C \cO_B^l.$$
In particular, the biorbit modules $\C \cO_A^{\bi}$ are  the sum of some Wedderburn components of $\C U_J$.
\end{Cor}

Taking this in conjunction with the fact that endomorphism rings of the left and right orbit modules generated by $[A]\in \hat{V}_J$ are isomorphic implies now:

\begin{Cor}\label{orbitsizelr} [\cite{super}, \cite{yan}]
Let $ A \in V_J$. Then $|\cO_A^r| = |\cO_A^l|$ and hence the number of right and left orbits contained in $\cO_A^{\bi}$  coincide. Moreover there exists $\kappa\in\N$ such that
$$
\C\cO_A^{\bi} \cong \bigoplus_{\kappa\text{ many copies}}\C\cO_A^r.
$$
Thus
$$ 
|\cO_A^{\bi}|=\frac{|\cO_A^r||\cO_A^l|}{|\cO_A^r\cap \cO_A^l|}=\frac{|\cO_A^r|^2}{|\cO_A^r\cap \cO_A^l|}
$$
and $$
\kappa = \frac{|\cO_A^r|}{|\cO_A^r\cap \cO_A^l|} .
$$

\end{Cor}

\begin{Defn}\label{superchar}
Let $J\subseteq \Phi^+$ be closed and $A\in V_J$. Then the $U_J$-characters afforded by $\C \cO_A^r$ is called \textbf{supercharacters} of $U_J$.
\end{Defn}

\begin{Remark}\label{supertheory}
A supercharacter theory for some finite group $G$ consists of a set partition of the collection of conjugacy classes, the unions of the parts called superclasses, and a set of pairwise orthogonal complex characters, called supercharacters such that every irreducible complex character of $G$ occurs as constituent in precisely one supercharacter. Moreover supercharacters are constant on superclasses and  the number of superclasses and supercharacters should coincide. The supercharacter theory was first introduced by Andr\'{e} [\cite{andre}] and Yan [\cite{yan}] for the unitriangular group $U_n(q)=U_{\Phi^+}$, and then generalized to $\F_q$-algebra groups by Diaconis and Isaacs  in [\cite{super}]. In particular, they produced the explicit character formula below. For the convenience of the reader, we shall also include a proof here:
\end{Remark}

\begin{Theorem}\label{SupCharVal} [\cite{super}, Results 5.6-5.8]
Let $J\subseteq \Phi^+$ be closed, and let $\chi_{_A}$ be the supercharacter, which is  afforded by $[-A]\C U_J$, where $A\in V_J$.  If $g\in U_J$ lies in the superclass $K$, then
\[
	\chi_{_A}(g)=\frac{|[A].U_J|}{|U_J.[A].U_J|}\sum_{[B]\in U_J.[A].U_J}
 [B]\circ f(g)
	=\frac{|[A].U_J|}{|K|}\sum_{h\in K}[A]\circ f(h).
\]
\end{Theorem}
\begin{proof}
For the first equation, we will mainly follow Yan's method in [\cite{yan}, Theorem 2.5]. For convenience we set $\cO^r=\cO_{-A}^r$ and $\cO^{\bi}=\cO_{-A}^{\bi}$. Note that the matrix representation of $g\in U_J$ with respect to the monomial basis $\{[B]\mid  B\in \cO^r\}$ is given as complex monomial matrix whose non-zero entries are given in Equation (\ref{monomial action}).
\begin{eqnarray}
\chi_{_A}(g)
&=&\sum_{\stackrel{[B]\in \cO^r}{[B].g=[B]}}[B]\circ f(g^{-1})=\sum_{[B]\in \cO^r}\langle [B].g, [B]   \rangle \, [B]\circ f(g^{-1})\nonumber\\
&=&\frac{1}{|U_J|}\sum_{[B]\in \cO^r}\bigg(\sum_{\tilde g\in U_J}
\theta\circ (B^*.g- B^*)f(\tilde g) \bigg)\theta \circ B^*\circ  f(g^{-1})\nonumber\\
&=&\frac{1}{|U_J|}\sum_{[B]\in \cO^r}\sum_{\tilde g\in U_J}
\theta\big( (B^*.g- B^*)\circ f(\tilde g) +B^*\circ f(g^{-1})\big),   \label{char1}
\end{eqnarray}
where $\langle\quad,\quad\rangle$ denotes the standard inner product of characters of $V_J$.
Moreover
\begin{eqnarray}
(B^*.g- B^*)\circ f(\tilde g) +B^*\circ f (g^{-1})
&=&(B^*.g- B^*)(\tilde g-E) +B^*(g^{-1}-E)\nonumber\\
&=&(B^*.g)(\tilde g-E)-B^*(\tilde g-E) +B^*(g^{-1}-E)\nonumber\\
&=&B^*(\tilde gg^{-1}-g^{-1})-B^*(\tilde g-E) +B^*(g^{-1}-E)\nonumber\\
&=&B^*(\tilde gg^{-1}-\tilde g)=\tilde g^{-1}.B^*(g^{-1}-E) \label{char2}
\end{eqnarray}
Inserting (\ref{char2}) into  (\ref{char1})  and applying \ref{3.11} we obtain:
\begin{eqnarray*}
\chi_{_A}(g)=\frac{1}{|U_J|}\sum_{[B]\in \cO^r}\sum_{\tilde g\in U_J}
\theta\big( \tilde g.B^*(g^{-1}-E)\big) =\frac{1}{|U_J|}\sum_{[B]\in \cO^r}\sum_{\tilde g\in U_J}
\tilde g.[B]( g^{-1}-E) 
\end{eqnarray*}
If $[B]$ runs through all elements of $\cO^r$ and $\tilde g$ runs through all elements of $U_J$, then $\tilde g.[B]$ will run through the elements of $\cO^{\bi}$, each with multiplicity
$
\frac{|\cO^r|\cdot|U_J|}{|\cO^{\bi}|}.
$
Thus we derive the first equation:
\begin{eqnarray}\label{char5}
\chi_{_A}(g)=\frac{|\cO^r|}{|\cO^{\bi}|}\sum_{[B]\in \cO^{\bi}}
 [B]\circ f(g^{-1})=\frac{|[A].U_J|}{|U_J.[A].U_J|}\sum_{[B]\in U_J.[A].U_J}
 [B]\circ f(g)
\end{eqnarray}
Let $t=f(g)=g-E\in V_J$ then the superclass $K$ containing $g$ is given as $E+U_J.t.U_J$. We consider the map $(u,v)\mapsto utv$ from $U_J\times U_J$ to $U_J.t. U_J$. This map is clearly surjective, and it is easy to see that all elements of $U_J.t.U_J$ are hit equally often. Each element of $U_J.t.U_J$, therefore, has the form $u.t.v$ for exactly $\frac{|U_J|^2}{|U_J.t.U_J|}$ ordered pairs $(u,v)$, and it follows that 
\begin{equation}\label{char3}
\sum_{u,v\in U_J}[A] (u.t.v)= \frac{|U_J|^2}{|U_J.t.U_J|}\sum_{x\in U_J.t.U_J}[A](x)
\end{equation}
Similarly, if we map  $U_J\times U_J$ to $U_J.[A]. U_J$ via $(u,v)\mapsto u^{-1}.[A].v^{-1}$, we see that each element of $U_J.[A].U_J$ is hit exactly $\frac{|U_J|^2}{|U_J.[A].U_J|}$ times, and this yields
\begin{equation}\label{char4}
\sum_{u,v\in U_J} u^{-1}.[A].v^{-1}(t)= \frac{|U_J|^2}{|U_J.[A].U_J|}\sum_{[B]\in U_J.[A].U_J}[B](t)
\end{equation}
The left hand side of (\ref{char3}) and (\ref{char4}) are the same, hence
\begin{equation}\label{char}
\frac{1}{|U_J.t.U_J|}\sum_{x\in U_J.t.U_J}[A](x)=\frac{1}{|U_J.[A].U_J|}\sum_{[B]\in U_J.[A].U_J}[B]\circ f(g)
\end{equation}
Inserting (\ref{char}) into  (\ref{char5}) we obtain:
\[
\chi_{_A}(g)=\frac{|[A].U_J|}{|U_J.t.U_J|}\sum_{x\in U_JtU_J}[A](x)=\frac{|[A].U_J|}{|K|}\sum_{h\in K}[A]\circ f(h)
\]
as desired.
\end{proof}

In the special case of $J = \Phi^+$ of the full unitriangular group $U = U_n(q)$ the vector space $V=V_J=\Lie(U)$ consists of all nilpotent upper triangular matrices. We define $A\in V$ to be a \textbf{verge}, if each row and each column of $A$ contains at most one non-zero entry. It is an easy exercise in linear algebra to show, that for every matrix $A\in V$ there exists a unique verge $v(A)$ obtained from $A$ by applying elementary upward row and left to right column operations. As a consequence there is a natural correspondence between the $U$-biorbits on $V$ and hence superclasses of $U$ and verges in $V$. 

To  obtain a similar description of the $U$-biorbits on $\hat V$ we apply restricted downward row and right to left column operations as described in illustrations \ref{3.10} and \ref{3.8} to produce from $[A]\in\hat V$ a unique element $[B] \in \V$ with $B\in V$ a verge, (which is in general different from $v(A)$). Thus $\{U.A.U+E|  A\in V \text{a verge}\}$ is the set of superclasses,  and $\{\chi_A| A\in V \text{a verge}\}$ the set of supercharacters of $U$.

For general pattern subgroups $U_J$ of $U$, $J\subseteq \Phi^+$ closed, the classification of superclasses and supercharacters as defined above is difficult, in fact, it is known in only a few special cases. We   embark next on the task, to find the supercharacters  for a special type of pattern subgroups, called column closed pattern subgroups. 
  
\section{Column closed pattern subgroups}\label{sec3}

\begin{Defn}\label{colclosed}
A subset $J$ of $\Phi^+$ is  \textbf{column closed} if  $J= \{(i,j)|1\leq i < j\leq n, j\in I\}$;
and  \textbf{row closed} if  $J= \{(i,j)|1\leq i < j\leq n, i\in I\}$ for some $I \subseteq \{1,\ldots,n\}$.
\end{Defn}

Thus $J$ is column (row) 
closed, if it arises by removing all positions $(i,j)\in\Phi^+$ of some columns (rows).  Note that column (row) 
closed subsets $J$ of $\Phi^+$ are complemented, that is,  the complementary set $J^c = \Phi^+\backslash J$ is again column (respectively row) closed, namely replacing $I \subseteq \{1,\ldots,n\}$ by $ \{1,\ldots,n\}\backslash I$. Thus $U_J$ and $U_{J^c}$ are complimentary pattern subgroups, that is $U_n(q) = U = U_JU_{J^c} = U_{J^c}U_J$. However,  this is not a semi-direct product in general.

In this paper we concentrate on column closed pattern subgroups. Indeed the row closed patterns behave quite different and there seems to be no  trivial transfer from the column closed to the row closed case, (c.f. \ref{row not easy}). However, we think there is a more subtle way for this using left orbits and hence Hom-spaces between right orbits for the column closed case to get information on the row closed case. This will be explored in a forthcoming investigation. 

Recall that for $J \subseteq \Phi^+$ column closed, the associated pattern subgroup of $U$ is $U_J$, its Lie algebra is $V_J$ and a two sided 1-cocycle $f: U_J \to V_J$ is given by $f:u \mapsto u-E\in V_J$ for $u\in U_J$. The linear characters of $V_J$ are given as $\hat{V}_J = \{[A]| A \in V_J\}$ (see \ref{lide}) If no ambiguities may arise, we drop superscripts $``J"$ and write for instance $U$ for $U_J$ and $V$ for $V_J$.

From now on in illustrations , we draw $[A]$ as strict upper triangle, omitting from matrix $A$ all superfluous zeros of the lower half.

\begin{Example}\label{ExLi}
 Let $n=6$ and $J$ be obtained by taking out column 3 and column 5  from $\Phi^+$, $A=\sum_{(i,j)\in J}A_{ij}e_{ij}\in V$. Then
\begin{center}
\begin{picture}(180, 100)
\put(20,45){$[A]=$}
\put(30,100){\line(1,0){100}}
\put(130,0){\line(0,1){100}}
\put(30,100){\line(1,-1){100}}

\put(50,88){\small $A_{12}$}
\put(69,88){$\ddagger$}
\put(79,88){\small $A_{14}$}
\put(99,88){$\ddagger$}
\put(110,88){\small $A_{16}$}

\put(69,71){$\ddagger$}
\put(79,71){\small $A_{24}$}
\put(99,71){$\ddagger$}
\put(110,71){\small $A_{26}$}

\put(79,54){\small $A_{34}$}
\put(99,54){$\ddagger$}
\put(110,54){\small $A_{36}$}

\put(99,37){$\ddagger$}
\put(110,37){\small$A_{46}$}

\put(110,20){\small$A_{56}$}

\put(160,55) {$A_{ij}\in \F_q, \,\forall\, (i,j)\in J$}
\put(170,35) {``\,$\ddagger$\,'' $\notin\, J$}
\end{picture}
\end{center}
denotes the lidempotent $[A]\in \hat V\subseteq \C V.$ 
\end{Example}

\begin{Defn}\label{flip}For $(i,j)\in J \subseteq \Phi^+$ we define the
\textbf{hook} $^J\!h_{ij}$ centered at $(i,j)\in J$ to be
$$^J\!h_{ij}= {}^J\! h_{ij}^{a} \cup {}^J\!h_{ij}^l \cup \{(i,j)\}, \text{ where }$$
$$ {}^J\! h_{ij}^a=\{(i,k)\in J\,|\,i<k<j\}\text{ called hook
\textbf{  arm} };$$
$$  {}^J\! h_{ij}^l=\{(k,j)\in J\,|\,i<k<j\}\text{ called hook
\textbf{leg} } ;$$
$(i,j)$ is called \textbf{hook center}.
For $J=\Phi^+$ we drop the index $J$, that is, $h_{ij}=^{\Phi\!^{^+}}\!\!\!\!h_{ij}, h_{ij}^a=^{\Phi\!^{^+}}\!\!\!\!h_{ij}^a$ and $ h_{ij}^l=^{\Phi\!^{^+}}\!\!\!\!h_{ij}^l$. It is clear that $\makebox[5pt][l]{$^{J}$} h_{ij}=h_{ij}\cap J$ for $J\subseteq \Phi^+$.
For $J\subseteq \Phi^+$ column closed and $(i,j)\in J$ the column $j$ is entirely contained in $J$ and hence $ ^J\! h_{ij}^l=h_{ij}^l$. Frequently we shall drop therefore the left superscript $J$ in denoting hook legs for $J$ column closed.

There is a bijection $\f=\f_{ij}$ of $h_{ij}\setminus \{(i,j)\}$ into itself, taking
$(i,k)$ to $(k,j)$ and $(k,j)$ to $(i,k)$ for $i<k<j$. We call $\f$ the \textbf{flip (map) centered} at $(i,j)$ and
note that $\f^2=1.$ So $\f$ is a bijection between  $h_{ij}^a$ and $h_{ij}^l$. 
We illustrate this by:
\end{Defn}

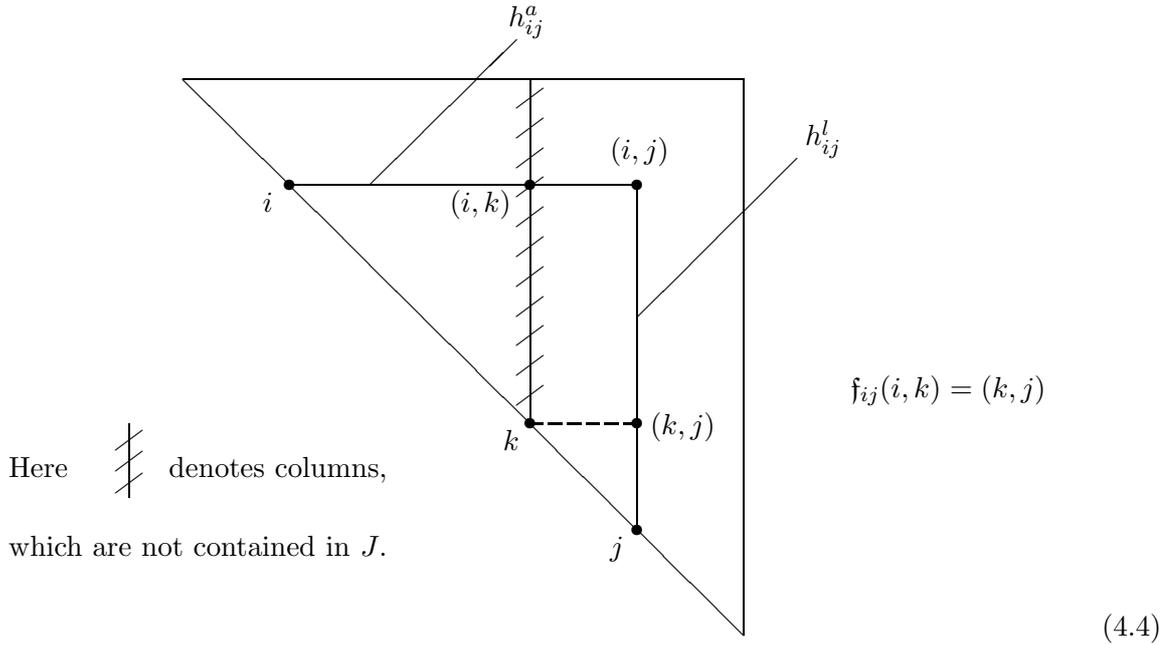
\begin{equation}\label{picnormalbijection}
\begin{picture}(200,230)
\put(130,170){\line(0,-1){130}}
\put(130,170){\line(-1,0){130}}
\put(-40,210){\line(1,-1){210}}

\put(130,170){\circle*{4}}
\put(0,170){\circle*{4}}
\put(130,40){\circle*{4}}
\put(130,80){\circle*{4}}
\put(90,170){\circle*{4}}
\put(90,80){\circle*{4}}
\put(135,76){\makebox{$(k,j)$}}
\put(120,180){\makebox{$(i,j)$}}
\put(60,160){\makebox{$(i,k)$}}
\put(-10,160){\makebox{$i$}}
\put(120,30){\makebox{$j$}}
\put(80,70){\makebox{$k$}}

\multiput(90,80)(7,0){6}{\line(1,0){5}}

\put(90,80){\line(0,1){130}}
\multiput(85,87)(0,11.2){11}{\line(4,3){10}}


\put(170,210){\line(0,-1){210}}
\put(170,210){\line(-1,0){210}}
\put(-105,60){\makebox{Here}}
\put(-60,80){\line(0,-1){28}}
\multiput(-65,70)(0,-8){3}{\line(4,3){10}}
\put(-45,60){\makebox{denotes columns,}}
\put(-105,30){\makebox{which are not contained in $J$.}}
\put(210,90){\makebox{$\f_{ij}(i,k)=(k,j)$}}

\put(130,120){\line(1,1){60}}
\put(193,185){$h_{ij}^l$}

\put(30,170){\line(1,1){55}}
\put(82,230){$h_{ij}^a$}
\end{picture}
\end{equation}

\begin{Defn}\label{jnormalrow}
Let  $J\subseteq \Phi^+$ be column closed, and suppose column $k$ is taken out, $1\leqslant k \leqslant n$, that is $(i,k)\notin J$ for $1\leqslant i<k\leqslant n$. Then $\row k =\{(k,j)\in J\,|\, k< j\leqslant n\}$ is said to be {\bf normal for $J$ or $J$-normal}.  The positions $(k,j)$ with $k< j$ are called {\bf $J$-normal} too.
\end{Defn}

\begin{Remark} \label{row not easy}
By \ref{3.13} the root subgroups of $U=U_J$, ($J$ column closed), corresponding to  the $J$-normal positions  act from the right by linear characters on the elements in $\hat V=\hat V_J$.  Similarly, if $J$ is row closed we may consider the left action of $U_J$ on $\hat V_J$ by restricted row operations (see \ref{3.9}). If row $k$  ($1\leqslant k\leqslant n$) is not contained in $J$, then we may call  analogously column $k$ ``$J$-normal'' . Then the  root subgroups corresponding to positions on column $k$  act from the left  by linear characters on the elements in $\hat V_J$. Indeed, the results of this paper carry over immediately to analogous results for row closed pattern subgroups $U_J$ acting from the left on $\hat V_J$.   However  in general  the  root subgroups corresponding to positions on column $k$   (row $k \nsubseteq J$) do {\bf not} act by linear characters from the right on the elements in $\hat V_J$. Thus, for the right action of $U_J$ on $\hat V_J$ by restricted column operations we do not have the notion of $J$-normality for row closed $J$. Instead one can see immediately, that the entries in column $k$ cannot be changed by the restricted column operations in this case and hence are constant on every $U_J$-orbit on $\hat V_J$. As a consequence the results for the column closed $J$  case and right action on $\hat V_J$ do not carry over immediately to the right action on $\hat V_J$ for row closed $J$.
\end{Remark}

\bigskip

Let  $[A]\in \hat{V}$. Next we will state a process to reduce the number of non-zero entries in $[A]$ by applying truncated column operations as described in Proposition \ref{3.11}: If A is the zero matrix, then $\cO_A^r(=\cO_A^l=\cO_A^{\bi})=\{[A]\}$. Otherwise, let column $j$ be the last (i.e. rightmost) non-zero column (belonging to $J$) in $A$. Let $z\in \F_q^*$ be the highest non-zero entry in column $j$ (i.e. first non-zero entry from top in column $j$) of $A$ and suppose it is at position $(i,j)$. Then acting by $x_{kj}(A_{ik}z^{-1})$ on $[A]$ from the right produces a lidempotent $[A']$ in $\hat{V}$ such that $A'$ coincides with $A$ at all positions not in column $k$ and $A'_{ik}=0$:

\begin{center}
\begin{equation}\label{pic4-3}
\begin{picture}(180, 140)
\put(130,0){\line(0,1){130}}
\put(130,0){\line(-1,1){130}}
\put(130,130){\line(-1,0){130}}

\put(100,90){$z$}
\put(100,100){$0$}
\put(101,109){$\vdots$}
\put(100,120){$0$}

\put(60,90){$A_{ik}$}

\put(70,60){\circle*{4}}
\put(58,48){$k$}
\multiput(70,61)(0,7){4}{\line(0,1){5}}
\multiput(70,61)(8,0){4}{\line(1,0){6}}

\put(35,95){\circle*{4}}
\put(24,84){$i$}
\multiput(35,95)(7,0){4}{\line(1,0){5}}

\put(103,60){\circle*{4}}
\put(106,53){\scriptsize$(k,j)$}
\multiput(102.5,57)(0,-7){4}{\line(0,-1){5}}

\put(103,27){\circle*{4}}
\put(92,15){$j$}

\put(102.5,132){\line(0,1){10}}
\put(102.5,142){\line(-1,0){35}}
\put(67.5,142){\vector(0,-1){10}}

\put(78,145){$-\alpha$}

\put(152,132){\line(0,1){10}}
\put(152,142){\line(-1,0){35}}
\put(117,142){\vector(0,-1){20}}
\put(140,125){zero columns}

\put(190,75){$z\in \F_q^*$}
\put(190,55){$\alpha=A_{ik}z^{-1}$}
\end{picture}
\end{equation}
\end{center}

Doing this using suitable elements of the root subgroups $X_{kj}$ in column $j$ for $i<k<j$ we obtain $[B]\in\hat{V}$ such that $B_{ik}=0$ for all positions $(i,k)\in J$ in row $i$ to the left of $(i,j)$.

Now consider $B $ and let column $l$ be the next column to the left of column $j$ (so $l< j$) which is not a zero column in $B$. Note that column $j$ coincides in $[A]$ and $[B]$ and column $l$ is contained in $J$. Moreover, the highest non-zero entry in column $l$ of $[B]$ cannot be located at position $(i,l)$, since $B_{ik}=0$ for all $j<k<i$ and $j<l<i$. Thus suppose the first non-zero entry in column $l$ is at position $(m,l)\in J$. Acting by suitable elements of root subgroups in column $l$ below position $(m,l)$  as above we may make all entries to the left of position $(m,l)$ and belonging to $J$ to zero. Note that by this column $j$ will not be changed any more. Proceeding like this, we will end up with a so called {right template} which we shall define now:

\begin{Defn}\label{template}
\begin{itemize}
\item [(1)] A lidempotent  $[A]\in \hat{V}$ is called {\bf (right) template}, if in $A$  to the left of the highest non-zero  entry in each non-zero column all entries at positions in $J$ are zero; (above that entry all entries in $A$ are zero by default). 
\item [(2)] A {\bf main condition} of a  template $[A]$ is a position $(i,j)$ such that $A_{ij}$ is the highest non-zero entry in column $j$ of $A$. Note that the main conditions of templates are in different rows by construction.  We denote the set of main conditions of the template $[A]$ by main$[A]$. The hooks centered at main conditions are called {\bf main hooks}. Intersections of hooks centered at main conditions are called {\bf main hook intersections}. Thus $(i,k)\in J$ is a main hook intersection if and only if there exist $1\leq i<j<k<l \leq n$ with $(i,k), (j,l) \in $ main$[A]$ and then $\{(j,k)\} =\,^J\!h_{ik}^l$ $\cap$ $^J\!h_{jl}^a$.
\item[(3)]The positions of a template $[A]$ strictly below main conditions not to the left of main conditions are called {\bf supplementary conditions} and their set is denoted by $\suppl(\p)$, where $\p = \main[A]$. Thus $A_{st}\not= 0$ implies $(s, t) \in \p$ or $ (s, t)\in \suppl(\p).$ Note that supplementary conditions are always contained in main hook legs. Indeed  $\suppl(\p)$ consists precisely of the positions on main hook legs different from hook intersections, because those are positions to the left of main conditions.
\end{itemize}
\end{Defn}

We have shown:
\begin{Lemma}\label{pattern exist}Every right orbit in $\hat{V}$ contains a   template.
\end{Lemma}

Next we shall prove that there is only one template in each right orbit.

\begin{Defn}\label{defofpstab} Let $[A]\in \hat{V}$. The {\bf (right) projective stabilizer} $\Pstab_{U}[A]$ of $[A]$ in $U=U_J$ is defined to be
\[
\Pstab_U[A]=\{u\in U\,|\, [A] u=\lambda_u [A], \, \text{for some }  \lambda_u \in \C ^*\}.
\]
Thus $\Pstab_U[A]$ acts on $\C [A]$ by the linear character $\tau: u\mapsto \lambda_u\in \C^*$. By general theory we have
\[
\C \cO_A^r=[A]\C U\cong \Ind^ {U}_{\Pstab_{U} [A]} \C_\tau,
\]
where $\C_\tau$ denotes the one dimensional $\C \Pstab_{U}[A]$-module affording $\tau$. The left projective stabilizer $\Pstab_{U}^{\ell}[A]$ is defined analogously.
\end{Defn}

\begin{Defn}\label{jnormalpattern}Let $[A]$ be a template and $\p=\main[A]$.  Recall that the positions of $J$ in row $i$ are called $J$-normal if column $i$ does not belong to $J$. The positions in $ \suppl(\p)$ which  belong to  $J$-normal rows are called {\bf normal (supplementary) conditions}, denoted by $\cN(\p, J)$ and the set $\cY(\p, J):=\suppl(\p) \setminus \cN(\p, J)$ consists of   the   supplementary conditions which are not $J$-normal, called {\bf non normal (supplementary) conditions} or {\bf  $\cY$-conditions} for short. Moreover $[A]$ is a {\bf normal template}, if  $A$ has only  non zero entries on main conditions or normal conditions.
\end{Defn}

\begin{Example}\label{hookintersection}
Let  $J=\Phi^+\setminus \{\text{column $k$}\}$, let $[A]$ be a template in $\V_J$ with main condition set $\p=\{(i,l), (j,m)\}$ and let $A_{il} = z_2$ and $A_{jm}=z_1$.
\begin{center}
\begin{picture}(250,160)
\put(-50,80){$[A]=$}

\put(0,160){\line(1,0){160}}
\put(0,160){\line(1,-1){160}}
\put(160,160){\line(0,-1){160}}

\put(20,140){\line(1,0){82}}
\put(102,137){$z_2$}
\put(105,135){\line(0,-1){80}}
\put(7,135){$i$}
\put(100,43){$l$}

\put(69,70){$k$}

\put(50,110){\line(1,0){82}}
\put(132,107){$z_1$}
\put(135,105){\line(0,-1){80}}
\put(32,107){$j$}
\put(126,16){$m$}
\put(105,110){\circle{4}}
\put(105,110){\line(2,1){70}}
\put(178,143){main hook intersection}

\put(80,80){\line(0,1){80}}
\multiput(75,85)(0,9){8}{\line(1,1){10}}
\multiput(80,80)(6.25,0){13}{\line(1,0){5}}

\put(135,80){\circle*{3}}
\put(105,80){\circle*{3}}
\put(105,80){\line(2,1){75}}
\put(135,80){\line(6,5){45}}
\put(182,115){normal supplementary conditions}
\end{picture}
\end{center}
Then row $k$ is normal, $(j,l)$ is the main hook intersection, and the elements in column $l$ south of $z_2$ and column m south of $z_1$ in $J$ except $(j,l)$ are the supplementary conditions. The supplementary conditions $(k,l)$ and $(k,m)$ are in addition normal, all others are $\cY$-conditions.
\end{Example}

\begin{Defn}\label{jplaces}
Let $[A]$ be a template and $\p=\main[A]$. 
Then the positions on the hook arms 
$^J\! h^a_{ij}, for (i,j)\in \p$ are called {\bf $J$-places} for $\p$. Denote the set of these positions by $\Pl(\p,J)$ So these are the positions in $J$ to the left of and on the rows of main conditions.
\end{Defn}

\begin{Remark}\label{feed}
Keep the notation in the definition above. For any $(i,j)\in \p$, our assumption that $J$ is column closed says that column $j$  and hence all positions on the hook leg $^J\!h_{ij}^l$ belong to $J$.  In particular, if $(i,k)\in J$  is on the hook arm  $^J\!h_{ij}^a$,  the flipped position $(k,j) \in$ $ ^J\!h_{ij}^l $  belongs to $J$ too and hence acting by the root subgroup $X_{kj}$ changes entries in column $k$ and especially at position $(i,k)$.  Direct inspection shows that the flip map $\f$ maps the $J$-places $\Pl(\p,J)\cap ^J\! h^a_{ij}$ on the main hook arm $^J\! h^a_{ij}$ bijectively to the set of non normal positions of $^J\! h^l_{ij} = h^l_{ij}$, that is the union of the set $\cY(\p, J)\cap  h^l_{ij}$  of  $\cY$- conditions and non normal main hook intersections  on the hook leg $ h^l_{ij}$.  
\end{Remark}

 By  \ref{pattern exist} for any $[B]\in \V$ we can find a template $[A]\in \cO^r_B\subseteq\V$. In view of   \ref{defofpstab} it suffices to determine $ \Pstab_{_{U\!_J}}[A]$ in order to find $|\cO_A^r|=\dim \C\cO_A^r$.

\begin{Defn}\label{cR}
Let $[A]\in \V$ be a template with $\p=\main[A]$.
Define $\cR=\cR(\p, J)\subseteq J$ to be the set of the following positions:
\begin{enumerate}
\item [1)] The main conditions in $\p$ and all positions above those.
\item [2)] All positions in  zero columns.
\item [3)] All positions of $J$ in normal rows.
\end{enumerate}
\end{Defn}
Note that this definition depends  only on the main conditions and $J$ itself.  We  indicate $\cR$ $J$-normal rows and $J$-places in the following illustration:

\begin{equation}\label{pic4-4}
\begin{picture}(350,160)
\put(170,160){\line(0,-1){180}}
\put(170,160){\line(-1,0){180}}
\put(-10,160){\line(1,-1){180}}

\put(80,70){\line(0,1){90}}
\multiput(75,78)(0,8){10}{\line(5,3){10}}

\put(132.9,90.9){$\square$}
\multiput(137.2,117)(0,6){5}{\circle*{1}}
\multiput(133,92)(0,8){3}{$\times$}
\multiput(133,142)(0,8){2}{$\times$}
\multiput(129,94)(-7.5,0){6}{\circle*{4}}
\multiput(71,94)(-7.5,0){2}{\circle*{4}}

\multiput(137,87)(0,-6){3}{\circle{4}}
\multiput(137,63)(0,-6){8}{\circle{4}}
\put(133,67){$\times$}

\multiput(109.5,45.8)(0,8){2}{$\times$}
\multiput(114,65)(0,5){15}{\circle*{1}}
\multiput(109.5,141)(0,8){2}{$\times$}

\multiput(102,70)(6,0){9}{\circle*{1}}

\multiput(85.1,67)(6.16,0){2}{$\times$}
\multiput(152.6,66.8)(6.7,0){2}{$\times$}

\put(73,20){\line(1,1){35}}
\put(50,10){\makebox{zero column}}

\put(158,66){\line(1,-1){28}}
\put(190,33){\makebox{$J$-normal row}}

\put(12,85){\line(1,1){60}}
\put(-15,72){\makebox{column  $\nsubseteq J$}}

\put(190,142){\makebox{$\times$ are positions in $\cR$}}
\put(190,124){$\square$}
\put(190,125){$\times$}
\put(202,125){\makebox{denotes the main condition}}
\put(202,110){\makebox{at position $(i,j)$}}
\put(195,95){\circle*{4}}
\put(202,92){\makebox{are $J$-places in $\Pl(\p,J)$, e.g. position $(i,k)$}}
\put(195,77){\circle{4}}
\put(202,74){\makebox{denote the positions in $J$ but not in $\cR$}}
\put(202,56){\makebox{e.g. position $(k,j)$}}

\put(114,36){{\circle*{2}}}
\put(110,24){$k$}

\put(137,13){{\circle*{2}}}
\put(133,0){$j$}

\put(56,94){{\circle*{2}}}
\put(45,90){$i$}
\end{picture}
\end{equation}

\bigskip

\begin{Prop}\label{pstabcR}
Let $[A]\in\V$ be a template with $\p=\main[A]$ and define $\cR=\cR(\p,J)$ as in \ref{cR}. Then $\cR$ is a closed subset of $\Phi^+$ and $\Pstab_U [A]=U_\cR$. Thus 
$\C \cO^r_A=[A]\C U\cong\Ind_{U_\cR}^{U} \C [A]$. In particular,  for $(i,j)\in \cR$, if the linear character of $X_{ij}$ acting on $[A]$ is non trivial then $(i,j)\in \main[A]$ or $i$ is a $J$-normal row.
\end{Prop}
\begin{proof}
 Directly inspecting the commutators (e.g. in illustration \ref{pic4-4}) one sees easily that $\cR$ is a closed subset of $\Phi^+$.

Let $(i,j)\in J$. If $(i,j)$ belongs to one of the three types of positions above, then by direct calculation using \ref{3.8} we see the root subgroup $X_{ij}\subseteq \Pstab_{U} [A]$ and the last statement of the argument holds. So we have shown $U_{\cR}\subseteq \Pstab_{U} [A]$.

To show the inverse inclusion we first order $J$ linearly as follows: Recall that the positions in  $\Pl(\p,J)$ are the  $J$-places  and are always located on main hook arms. We order this set along columns left to right and such that for each column higher places come first. Thus for $(i,j), (k,l)\in \Pl(\p,J)$ we set $(i,j)\leq (k,l)$ if either $j< l$ or $j=l$ and $i\leq k$. This defines a linear order on the set $\Pl(\p,J)$ of $J$-places which consists of all positions of the main hook arms. Using the various flip maps centered at main conditions  we can by  \ref{feed} transfer this ordering to the set of all non normal positions on the main hook arms which  we know is $J\setminus\cR$.  

Let $(r,j)\in\p$ and suppose $(i,j)\in ^J\!h_{rj}^l $ is not normal, that is $i\in J$. Then $(i,j)\in J\setminus\cR$ and $(r,i)= \f(i,j) \in \Pl(\p,J)$. The root subgroup $X_{ij}$ acting from the right on lidempotents will change  entries only in column $i$. Thus, if $ (k,l)\in \Pl(\p,J)$ comes strictly earlier than $(r,i)$  in our chosen order, then either $l< i$ or $l=i$ and $k<r$.  Let $[B]\in\cO_A^r$ and suppose, columns $j$ of $A$ and $B$ coincide and hence $B_{sj}=0$ for $1\le s<r$, since $(r,j)\in \p$. Let $\alpha\in\F_q$ and set $[C]=[B]x_{ij}(\alpha)$. Then $C_{si}=B_{si}$ for all $1\le s<r$. Moreover all columns of $B$ and $C$ coincide except possibly the $i$th one and hence $B$ and $C$ coincide in particular in all positions of $\Pl(\p,J)$ coming earlier than $(r,i)$.

We order $J$ such that all roots of $\cR$ come first and then the roots in $J\setminus\cR$ in the ordering constructed above. By \ref{3.1} every element $g$ of $U_J$ may be uniquely written as:
$$
g=\prod_{(a,b)\in J}x_{ab}(\alpha_{ab})
$$
where   the product is taken in the linear ordering of $J$ constructed above. 
Thus if $g\in \Pstab_{U} [A]$  we may write $g$ as product $g=xy$, where $x\in U_\cR\subseteq \Pstab_{U} [A]$. Now $x$ fixes $[A]$  up to a scalar yielding $[A] y=\lambda\, [A]$ for some $\lambda \in \C^*$, since $g\in \Pstab_{U} [A]$. By construction, $y$ is a product of elementary matrices $x_{ij}(\alpha)$ with $(i,j)\in J\setminus \cR$, hence $(i,j)\in \J h_{r j}^l$ for  some main condition $(r, j)\in \p$. Note that $(i,j)$ is not contained in a normal row, that is $i\in J$. 
Suppose $y\neq 1$. Writing 
$$
y = \prod_{(a,b)\in J\setminus \cR}x_{ab}(\alpha_{ab})
$$

in the order constructed above suppose, that the first factor with $\alpha_{ab} \neq 0$ and hence $x_{ab}(\alpha_{ab})\neq 1$ is at position $(a,b) = (i,j)\in  \J h_{r j}^l$. Then 
$$
[A].x_{ij}(\alpha_{ij}) = [B], \text{ where } B_{ri}=A_{ri} -\alpha_{ij}A_{rj} \neq A_{ri}.
$$
But  entry $B_{ik}$ will never be changed by the subsequent factors in $y$, contradicting $y\in \Pstab_{U}[A]$. Thus $y=1$ and $g=x\in U_\cR$, as desired.
\end{proof}

\begin{Cor}\label{unique one}
Each right $U_J$-orbit contains exactly one template and hence those classify the right orbits of $U_J$ acting on $\hat{V}_J$.
\end{Cor}
\begin{proof}
The existence has been showed in \ref{pattern exist} so we only need to prove the uniqueness.
 Let $[A]\in \V$ be a template with $\p=\main [A]$. Define $\cR=\cR(\p, J)$ as in \ref{cR}. 

Let $I = J\setminus \cR$. So $I$ consists of all positions in $J$ which belong to the main hook legs but are not $J$-normal positions. Then $\{A+E\mid A\in V_J, \supp(A)\subseteq I\}$ give a set of coset representatives of $U_{\cR}$ in $U_J$ (c.f. the illustration \ref{pic4-4} above). Applying the flip maps centered at main hooks (see \ref{flip}) we see, that these positions are in bijection with  the collection of $J$-places on the hook arms, compare \ref{jplaces}. Since a template must have zeros on all $J$-places we conclude that $[A]$ is the unique template in its right orbit.
\end{proof}
 
We can now provide a description of the elements in the $U$-right orbits on $\V$:

\begin{Prop} \label{original notation for right}
Let $[A]\in \V$ be a template with $\main[A] = \p$ and let $\Pl(\p,J)$  be the set of $J$-places of $[A]$. Then for any choice of coefficients $\alpha_{ij}\in\F_q$ for $(i,j)\in\Pl(\p,J)$, there  exists a unique $[B]\in \cO_A^r$ with $B_{ij}=\alpha_{ij}$ for all $(i,j)\in\Pl(\p,J)$. Thus in particular $|\cO_A^r| = q^{|\Pl(\p,J)|}$ and for $[B]\in \cO_A^r$ the entries $B_{kl}$, where $(k,l)\in J$ but $(k,l)\notin  \Pl(\p,J)$,  depend linearly on $B_{ij}$ with $(i,j)\in  \Pl(\p,J)$.
\end{Prop}
\begin{proof}
By \ref{pstabcR} $\Pstab_U [A]=U_\cR$, where $\cR\subseteq J$ is defined in \ref{cR}. We order $J$ as in the proof of  \ref{pstabcR}. Obviously, when acting by $y$ on $[A]$ the first factors of $y$ belonging to $U_\cR$ will contribute only a non-zero scalar and can be ignored. A factor $x_{kj}(-\alpha), \alpha\in\F_q$ with $(k,j) \in \J h_{i j}^l, (i,j) \in \p$ will insert an entry $\alpha$ at the flipped position $\f_{ij}(k,j) = (i,k)\in \Pl(\p,J)$ and this entry $\alpha$ at that position $(i,k)\in \Pl(\p,J)$ will never be changed by the subsequent factors in $y$ (comp. proof of \ref{pstabcR}). From this the proposition follows immediately.
\end{proof}

So far we have achieved  the classification of $U_J$-right orbits on $\V_J$ by templates and described the elements in these orbits. By \ref{equalororth}  we know, that two different right orbit modules are either isomorphic, affording equal supercharacters or are disjoint  affording orthogonal ones. Thus   in order to classify the supercharacters of $U_J$ we need to determine, which of our right orbits are isomorphic and which are not. This will be done inspecting $\Hom$-spaces between orbit modules and make use of \ref{HomSpace}.

\section{A guiding example}
Throughout this section $J\subseteq\Phi^+$ is column closed and $U=U_J, V=V_J$. The next lemma shows, that for templates in $\V$ we may always assume that there are no non-zero $\cY$-conditions:
 
\begin{Lemma}\label{yiso}
Let $[A]$ and $[B]$ be two templates in $\V$. If $\main[A]=\main[B]$, then $\Pstab_{U_J} ([A])=\Pstab_{U_J} ([B])$. In addition, if the normal supplementary conditions of $[A]$ and $[B]$  coincide, then the  linear characters of their projective stabilizers acting  on $[A]$ (resp. $ [B]$) from right coincide as well. Thus templates $[A]$ and $[B]$ differing only in $\cY$-conditions generate different right orbits $\cO_A^r\neq\cO_B^r$ but isomorphic orbit modules $\C\cO_A^r\cong\C\cO_B^r$.
\end{Lemma}
\begin{proof}
The first statement follows directly from \ref{cR}. If $A$ and $B$  only differ on $\cY$-conditions, then
 $U_\cR=\Pstab_{U_J} ([A])=\Pstab_{U_J} ([B])$ acts on $[A]$ and $[B]$ from right by the same linear character since $\cR$ does not contain the non normal supplementary conditions. \end{proof}

Thus to classify the right orbit modules, we only need to deal with those modules generated by normal templates. But the trouble is, that the converse of \ref{yiso} is not true. Two different normal templates in $\V$ can still generate isomorphic right orbit modules.  The following example describes a situation, where we can remove a non-zero normal supplementary condition from a template and still obtain an isomorphic orbit module:

\begin{Example}\label{basiceg}
Let $\alpha, \beta\in \F_q^*=\F_q\setminus{\{0\}}$, $J=\Phi^+\setminus \{\text{column $k$}\}$ and let
\begin{center}
\begin{picture}(250,160)
\put(-50,80){$[A]=$}
\put(170,80){be a normal template in $V_J$, then}

\put(0,160){\line(1,0){160}}
\put(0,160){\line(1,-1){160}}
\put(160,160){\line(0,-1){160}}

\put(20,140){\line(1,0){80}}
\put(102,137){$z_2$}
\put(105,130){\line(0,-1){45}}
\put(102,75){$\beta$}
\put(105,72){\line(0,-1){17}}
\put(7,135){$i$}
\put(100,43){$l$}

\put(69,70){$k$}

\put(50,110){\line(1,0){80}}
\put(132,107){$z_1$}
\put(135,100){\line(0,-1){15}}
\put(132,76.8){$\alpha$}
\put(135,72){\line(0,-1){46}}
\put(32,107){$j$}
\put(126,16){$m$}

\put(80,80){\line(0,1){80}}
\multiput(75,85)(0,9){8}{\line(1,1){10}}
\multiput(80,80)(6,0){4}{\line(1,0){4}}
\multiput(109,80)(6,0){4}{\line(1,0){4}}
\end{picture}
\end{center}
 $x_{ij}(-z_1\beta \alpha^{-1}z_2^{-1}).[A].x_{lm}(\beta\alpha^{-1})=[B]$ where $B_{kl}=0$ and $B_{ab}=A_{ab}$ for all $(k,l)\neq(a,b)\in J$. Thus $[B]$ is also a normal template in $V_J$ and the right orbit module generated by it is isomorphic to $\C \cO_A^r$.
\end{Example}

In general if $[A]$ is a normal template, then all lidempotents in $\cO_A^l$ coincide with $[A]$ at all normal rows, since entries there cannot be changed by restricted row operations. Thus if $[B]\in \cO_A^l\cap \cO_A^r$, then $B$ and $A$ coincide at all normal rows and are distinct at most at non normal main hook intersections. Whereas the entries in $J$-normal positions cannot be changed by restricted row operations, this can be done by applying restricted column operations. In the following guiding example we illustrate how homomorphisms between orbit modules can be exhibited. 

\begin{Example}\label{bigeg}
In the following guiding example we demonstrate the basic ideas, how to detect for a normal template $[A]\in \V$  ``enough" elements in the intersection $\cO_A^r\cap \cO_A^l$ together with templates $[B]\neq [A]$ in $\cO_A^{\bi}$ such that equations \ref{orbitsizelr} are satisfied. 

To simplify notation for $[A]$ we have labelled the rows and columns containing non-zero entries in the illustration below simply by consecutive natural numbers. Thus let the normal template $[A]$ with main conditions $\p$ be given as follows:

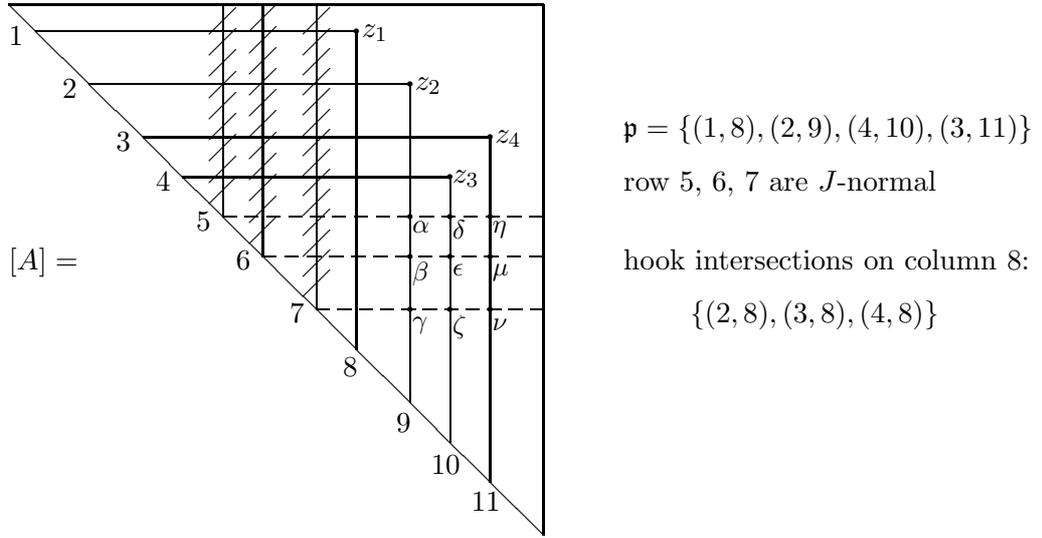
\quad
\begin{center}
\begin{picture}(350,195)
\put(0,100){$[A]=$}
\put(0,200){\line(1,0){200}}

\put(0,200){\line(1,-1){200}}
\put(200,200){\line(0,-1){200}}

\put(10,190){\line(1,0){120}}
\put(130,190){\line(0,-1){120}}
\put(0,185){1}
\put(130,190){\circle*{2}}
\put(132,188){$z_1$}

\put(30,170){\line(1,0){120}}
\put(150,170){\line(0,-1){120}}
\put(20,165){2}
\put(150,170){\circle*{2}}
\put(152,168){$z_2$}

\put(50,150){\line(1,0){130}}
\put(180,150){\line(0,-1){130}}
\put(40,145){3}
\put(180,150){\circle*{2}}
\put(182,148){$z_4$}

\put(65,135){\line(1,0){100}}
\put(165,135){\line(0,-1){100}}
\put(55,130){4}
\put(165,135){\circle*{2}}
\put(166,133){$z_3$}

\put(115,200){\line(0,-1){115}}
\multiput(115,85)(8,0){11}{\line(1,0){5}}
\multiput(110,90)(0,7.6){14}{\line(1,1){10}}
\put(105,80){7}

\put(95,200){\line(0,-1){95}}
\multiput(95,105)(8.3,0){13}{\line(1,0){5}}
\multiput(90,110)(0,8){11}{\line(1,1){10}}
\put(85,100){6}

\put(80,200){\line(0,-1){80}}
\multiput(80,120)(8.2,0){15}{\line(1,0){5}}
\multiput(75,125)(0,8){9}{\line(1,1){10}}
\put(70,115){5}

\put(150,120){\circle*{2}}
\put(151,114){\small $\alpha$}

\put(150,105){\circle*{2}}
\put(151,96){\small $\beta$}

\put(150,85){\circle*{2}}
\put(151,78){\small $\gamma$}

\put(165,120){\circle*{2}}
\put(166,112){\small $\delta$}

\put(165,105){\circle*{2}}
\put(166,98){\small $\epsilon$}

\put(165,85){\circle*{2}}
\put(166,76){\small $\zeta$}

\put(180,120){\circle*{2}}
\put(181,114){\small $\eta$}

\put(180,105){\circle*{2}}
\put(181,98){\small $\mu$}

\put(180,85){\circle*{2}}
\put(181,77){\small $\nu$}

\put(125,60){8}

\put(145,39){9}

\put(158,23){10}

\put(173,9){11}

\put(230,150){$\p=\{(1,8), (2,9), (4,10), (3,11)\}$}

\put(230,130){row 5, 6, 7 are $J$-normal}

\put(230,100){hook intersections on column 8:}

\put(255,80){$\{(2,8), (3,8), (4,8)\}$}
\end{picture}
\end{center}

Obviously if $\alpha=\beta=\gamma=0$, then $[B]=[A].x_{89}(-z_2^{-1}\omega)=x_{12}(-z_1^{-1}\omega).[A]$ coincides with $[A]$  except at position $(2,8)$ with $B_{28}=\omega$ for all $\omega\in \F_q$. So $[B]=[B_\omega]\in \cO_A^r\cap\cO_A^l$,  i.e. we have $q$ distinct lidempotents in $\cO_A^r\cap \cO_A^l$.
\end{Example}

 Suppose in the picture above there exist $\rho, \sigma\in \F_q$ with 
\begin{equation}\label{lincomb}
\begin{pmatrix}\alpha\\ \beta\\ \gamma\end{pmatrix}=
\sigma\begin{pmatrix}\delta\\ \epsilon\\ \zeta\end{pmatrix}+\rho\begin{pmatrix}\eta\\ \mu\\ \nu\end{pmatrix},
\end{equation}
then one checks easily that for all $\omega\in \F_q$
\begin{eqnarray*}
&&[A].x_{89}(-z_2^{-1}\omega)x_{8,10}(z_2^{-1}\omega\sigma)x_{8,11}(z_2^{-1}\omega\rho)\\&=&
x_{12}(-z_1^{-1}\omega)
x_{13}(z_1^{-1}z_2^{-1}\omega\sigma z_3)x_{14}(z_1^{-1}z_2^{-1}\omega\rho z_4).[A].
\end{eqnarray*}
Thus if we let $\omega$ run through $\F_q$ we have found  $q$ many elements in $\cO_A^r\cap \cO^l_A$.

But there is more about this construction, it produces as well isomorphisms between different orbit modules as we shall explain now:

Consider the $3\times 2$-matrix 
$B=\begin{pmatrix}
\delta     &\eta\\
\epsilon  &  \mu\\
\zeta       &  \nu
\end{pmatrix}$. 
Let $a,b\in \F_q$ and act by $u_1=x_{8,10}(-a)x_{8,11}(-b)$ on $[A]$. Then $[C]=[A].u_1$ satisfies: $C$ coincides everywhere with $A$ except in column $8$ which is given by $C_{18}=z_1, C_{38}=bz_4, C_{48}=az_3$ and 
\[
\begin{pmatrix}
C_{58}\\C_{68}\\C_{78}
\end{pmatrix}=
\begin{pmatrix}
A_{58}\\A_{68}\\A_{78}
\end{pmatrix}
+B\begin{pmatrix}a\\b\end{pmatrix} 
\]
Suppose $B$ has rank two and the first two rows of $B$  are linearly independent. Then given any $\lambda, \rho\in\F_q$ we can always find $a,b\in \F_q$ such that $C_{58}=\lambda$ and $C_{68}=\rho$. Moreover $C_{78}$ depends then linearly on $\lambda$ and $\rho$. Thus once $\lambda, \rho$ are chosen, then $C_{78}$ is fixed as well.
Now acting by $u_2=x_{13}(z_1^{-1}bz_4)x_{14}(z_1^{-1}az_3)$ from the left on $[C]$ will set positions $(3,8)$ and $(4,8)$ back to zero. We see therefore:
\[
u_2.[A].u_1=[\tilde A]
\]
where $[\tilde A]\in \hat V$ is the template which agrees with $[A]$ at all positions except $3$ supplementary normal conditions at positions $(5,8),(6,8),(7,8)$, which are given as $\lambda, \rho$ and the third term dependent on $\lambda, \rho$. Then $\C\cO_A^r\cong \C\cO_{\tilde A}^r$ and by the free choices of $\lambda$ and $\rho$, we have $q^2$ many different right orbits shown to be isomorphic.

Recall that by \ref{orbitsizelr}, \ref{original notation for right}  and \ref{HomSpace}
\begin{equation}\label{5.5}
\kappa|\cO_A^r\cap \cO_A^l|=|\cO_A^r|=q^m
\end{equation}

where $\kappa$ is the number of right orbits contained in  the biorbit $\cO_A^{\bi}$ generated by $[A]$, a $\C$-basis of $\End_{\C U_J}(\C\cO_A^r)$ is given by $\cO_A^r\cap \cO_A^l$ and $m$ is the number of $J$-places $\Pl(\p,J)$ of $[A]$ which in turn are precisely the positions on the main hook arms contained in $J$.  As we shall see there is a local  "mainhook-wise" version of formula \ref{5.5} which multiplied up over all main hooks gives the global formula \ref{5.5}. 

We observe that in example \ref{bigeg} the $J$-places on the hook arm centered at $z_1$ are the positions $\neq (1,5), (1,6), (1,7)$  and are in bijection by the flip map \ref{flip} with the non normal positions on the hook leg centered at the main condition $(1,8)$. These are all positions on column $8$ south of $(1,8)$ except the normal positions $(5,8), (6,8)$ and $(7,8)$. All of those, except three main hook intersections $(2,8), (3,8)$ and $(4,8)$, are $\cY$-conditions and hence contribute $q^a$ to the local $\kappa$, where $a$ is the number of these $\cY$-conditions in column $8$. The flipped positions $(1,2),(1,3)$ and $(1,4)$ to the main hook intersections $(2,8), (3,8)$ and $(4,8)$ respectively are $J$-places in $\Pl(\p,J)$ which contribute a summand of $3$ to $m$ hence a factor $q^3$ to $q^m$. 

Now the constellation of the normal supplementary conditions to the right of column 8 denoted by Greek letters determines hook intersection $(1,2)$ to contribute a factor $q$ to automorphisms of $\cO_A^r$, that is to $\cO_A^r\cap \cO_A^l$. The normal positions $(5,8)$ and $(6,8)$ in exchange for the remaining two hook intersections (2,8) and (3,8) not yet accounted for, contribute factors $q$ to $\kappa$ in \ref{5.5}. Thus for the $J$-places in row 1, that is the number of left $J$-places in column $8$ equation \ref{5.5} reads now
\[
q^{|\cY(\p,J)|_8}\cdot q^2 \cdot q=q^{|\text{left} J\text{-places}|_8}
\]
where $|\quad|_8$ denotes the local quantities on column $8$.

How this ``exchange'' of some (non normal) main hook intersections with normal supplementary positions happens is explained in the next section.

\section{The $U_J$-biorbits on $\hat V_J$}
Recall that $M_n(q)= M_n(\F_q)$ denotes the $\F_q$-space of $n\times n$-matrices with coefficients in $\F_q$. Similarly we shall denote the $\F_q$-space of $a\times b$-matrices over $\F_q$ simply by $M_{a\times b}(q)$  for any natural numbers $a,b$.

Let again $J\subseteq \Phi^+$ be column closed, $U=U_J, V=V_J$ and let $[A]\in \hat V$. Recall (see  \ref{3.11}) that
$[A].u=[\pi_{_J}(Au^{-t})] \quad \text{for } u\in U$,
and hence, by \ref{3.8}, the root subgroups $X_{bj}$, acting from the right on $[A]$, change entries in column $b$ of $[A]$ only.  Thus if 
\begin{equation}\label{6.1}
g=\prod_{k=b+1}^n x_{bk}(-\alpha_k) \in X_{b,b+1}\times X_{b,b+2}\times\ldots \times X_{bn},
\end{equation}
 with $ \alpha_{b+1}, \ldots, \alpha_n\in \F_q  \text{ and } \alpha_k=0 \text{ for } (b,k)\notin J$, we have by \ref{3.11}
\begin{equation}\label{6.2}
[A].g=[\pi_{_J}(Ag^{-t})] =[\pi_{_J}(A\cdot\prod_{k=b+1}^n x_{kb}(\alpha_k))].
\end{equation}
 Let $\tilde v=(0,\ldots,0,\alpha_{b+1}, \ldots,\alpha_n)^t\in \F_q^n$
and $\tilde w=A\tilde v\in \F_q^n$ then by direct inspection we have
$Ag^{-t}=\tilde B\in M_n(q)$, where $\tilde B$ coincides with $A$ on all columns except the $b$-th one which is given as $\tilde B_b=A_b+\tilde w$. Here $\tilde B_b$ respectively $A_b$ denotes the $b$-th column of $\tilde B$ and $A$ respectively. Now $\tilde w=(\beta_1,\cdots,\beta_n)^t\in \F_q^n$ and hence $\pi_{_J}(\tilde B)=B\in V$, where
\[
B_b=A_b+(\beta_1, \beta_2,\cdots,\beta_{b-1},0,0, \cdots, 0)^t.
\]
But for $1\leqslant l \leqslant n$ we have $\beta_l=\sum_{t=b+1}^n A_{lt}\alpha_t$. Thus we shall focus on a $(b-1)\times(n-b)$-submatrix $A(b)$ obtained from $A$ by deleting columns $1,\ldots, b$ and rows $b,\ldots, n$ from $A$ that is:

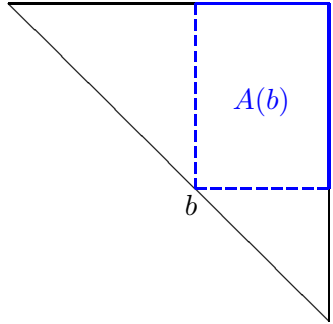
\begin{center}
\begin{picture}(120,120)
\put(0,120){\line(1,0){120}}
\put(120,0){\line(0,1){120}}
\put(0,120){\line(1,-1){120}}


\multiput(70,50)(5.7,0){9}{\thicklines\color{blue}\line(1,0){4}}
\multiput(70,50)(0,6){12}{\thicklines\color{blue}\line(0,1){4}}
\put(120,120){\thicklines\color{blue}\line(0,-1){70}}

\put(120,120){\thicklines\color{blue}\line(-1,0){50}}

\put(84,80){\thicklines\color{blue}$A(b)$}
\put(66,40){$b$}
\end{picture}
\end{center}

Now suppose in addition that $[A]\in \hat V$ is a template with main conditions $\p\subseteq J$. Let again $1\leqslant b\leqslant n$ and suppose column $b$ belongs to $J$. Since $[A]$ is a template, only columns containing a main condition have non-zero entries, and hence the same is true for $A(b)$. Let
\begin{equation}\label{6.3a}
\Gamma_b=\{(i,j)\in J\,|\, 1\leqslant i<b<j\leqslant n\},
\end{equation}
and  set $\Gamma_b\cap \p=\{(i_1, j_1), \ldots, (i_l, j_l)\,|\, b<j_1<\cdots<j_l\}$. Then the only non zero columns of $A(b)$ are columns $j_1, \ldots, j_l$ and contain main conditions which are in pairwise different rows. As a consequence, omitting all columns $c\notin \{j_1, \ldots, j_l\}$ from $A(b)$ produces $(b-1)\times l$-matrix $\tilde A(b)$ of rank $l$. Note that for $b<k\leqslant n$ with $k\notin \{j_1, \ldots, j_l\}$, $X_{bk}$ acts as identity on $[A]$ adding a multiple of the zero column $k$ to column $b$. Thus in \ref{6.2} we may assume that 

\begin{equation}\label{6.3}
g=\prod_{\nu=1}^l x_{bj_\nu}(-\alpha_\nu) \quad\text{ with } \alpha_{1}, \ldots, \alpha_l\in \F_q.
\end{equation}

Matrix multiplication from the left by $\tilde A(b)$ defines an injective $\F_q$-homomorphism from $\F_q^l$ into $\F_q^{b-1}$ whose image is spanned by the column vectors of $\tilde A(b)$.
Now we may reformulate \ref{6.2} as:

\begin{Remark}\label{6.4}
Define $w=\tilde A(b)v\in \F_q^{b-1}$ with $v=(\alpha_{1}, \cdots, \alpha_l)^t\in \F_q^{l}$. Then 
\[
[A].g=[B]\in \hat V
\]
where $B$ coincides with $A$ at all positions except those of column $b$, and the $i$-th entry of column $b$ for $1\leqslant i\leqslant b-1$  is given as $\tilde A_{ib}+w_i$.
\end{Remark}

As in \ref{5.5}  let $|\cO_A^r|= q^m$, where  $m$ is the number  $|\Pl(\p, J)|$ of $J$-places of $\p$ given as number of positions in $J$ of the main hook arms $^J\!h_{ij}^a, (i,j)\in \p$ by \ref{jplaces}. The flip maps $f_{ij}$ of \ref{flip} attached to $(i,j)\in \p$ define bijections between the $J$-places of $^J\!h_{ij}^a$ and the non $J$-normal positions of $^J\!h_{ij}^l =h_{ij}^l $.

Now all positions on $h_{ij}^l$ which are not $J$-normal and are not main hook intersections are $\cY$-positions (see \ref{jnormalpattern} and \ref{hookintersection}), each contributing a factor $q$ to $\kappa$ in \ref{5.5}. 


If all main hook intersections on $h_{ij}^l$ are normal, $f_{ij}$ is a bijection between the $J$-places on $^J\!h_{ij}^a$ and the $\cY$-conditions on $h_{ij}^l$, and we are done in this case. Hence assume that we deal with a column $b$ of $[A]$ containing at least one non normal main hook intersection. Then column $b$ contains a main condition of $[A]$, say $(a,b)\in \p$. 
In view of \ref{yiso} we may assume that $[A]$ is a normal template. Thus besides the main conditions for $[A]$ all non zero entries of $A$ are in $J$-normal rows and hence on normal supplementary positions, see \ref{jnormalpattern}.


Recall that $(a,b)\in \p$ is the main condition on column $b$ of $[A]$, and that column $b$ contains in addition at least one non normal hook intersection. Thus there exists $(c,d)\in \p$ with $a<c<b<d$ and column $c$ contained in $J$. Define
\begin{equation}\label{6.5}
\Gamma_{ab}=\{(i,j)\in J\,|\, a<i<b<j\}\subseteq \Gamma_b
\end{equation}
where $\Gamma_b$ is defined as in \ref{6.3a}. Then

\begin{equation}
\begin{picture}(120,120)
\put(0,120){\line(1,0){120}}
\put(120,0){\line(0,1){120}}
\put(0,120){\line(1,-1){120}}

\put(20,100){\line(1,0){45}}

\put(70,50){\line(1,0){50}}
\put(70,50){\line(0,1){45}}
\put(67,97){$z$}
\put(10,97){$a$}

\put(120,95){\thicklines\color{blue}\line(0,-1){41}}
\put(75,95){\thicklines\color{blue}\line(0,-1){41}}

\put(120,95){\thicklines\color{blue}\line(-1,0){45}}
\put(120,54){\thicklines\color{blue}\line(-1,0){45}}

\put(86,70){\thicklines\color{blue}$\Gamma_{ab}$}
\put(66,40){$b$}
\put(105,50){\line(2,-1){30}}
\put(135,28){row $b\ni g$}
\put(72,105){\line(2,1){50}}
\put(125,130){main condition}
\end{picture}
\end{equation}

Let $(i,j)\in \p\cap \Gamma_b$. If $i<a$, then $^J\!h_{ab}\cap\!\,^J\!h_{ij}=\emptyset$  and $X_{bj}$ acting on $[A]$ adds a multiple of column $j$ (containing the main condition $(i,j)$) to column $b$. Thus if we act by $x_{bj}(\alpha)$ with $\alpha\neq 0$, we insert $-\alpha A_{ij}$ into the position $(i,b)$ to the north of the main condition $(a,b)$, which can be possibly removed only on the prize of inserting an even more northern non-zero entry. This suggests that we better do not use the subgroup $X_{bj}$ in our consideration. 

\begin{Setup}\label{6.8}
$[A]\in \hat V$ is normal with $\main A=\p$ and $(a,b)\in \p$. 
$\Gamma_b, \Gamma_{ab}\subseteq J$ are defined as in \ref{6.3a} and \ref{6.5}. Since we may ignore main conditions in rows $< a$ anyway, we may set  $\p\cap \Gamma_{ab}=\{(i_1, j_1),\ldots,(i_l, j_l)\}$ again. Suppose $\p\cap \Gamma_{ab}$ contains  $r$ many normal main conditions, then we assume that $(i_1, j_1), \ldots, (i_r, j_r)$ are those. 
We define $\tilde A(a,b)\in M_{(b-1)\times l}(q)$ to be the matrix obtained from  $\tilde A(b)$  by deleting all columns containing  a main condition not contained in $\Gamma_{ab}$. Note that the rank of $\tilde A(a,b)$ is $l$ since the $l$ columns contain each a main condition with non-zero entry in $l$ different rows and all entries to the north of main conditions are zero. Moreover row 1 up to row $a$ of $\tilde A(a,b)$ are zero rows. So multiplication by $\tilde A(a,b)$ defines an injective $\F_q$-homomorphism from $\F_q^l$ into $\F_q^{b-1}$. 
For $v=(\alpha_1,\ldots,\alpha_l)^t\in \F_q^l$ let $g(v)=\prod_{\nu=1}^l x_{bj_\nu}(-\alpha_\nu)\in U$. Acting by $g(v)$ from the right on $[A]$ adds $\tilde A(a,b)v=(\beta_1,\ldots,\beta_{b-1})\in \F_q^{b-1}$ to column $b$ \big(from position $(1,b)$ to position $(b-1,b)$\big).\hfill$\square$
\end{Setup}

\begin{Prop}\label{6.9}
Let $a<k_1<\cdots<k_l<b$ be such that rows $k_1,\cdots, k_l$  of $\tilde A(a,b)$ are linearly independent and hence are a basis of the row space of $\tilde A(a,b)$. Then for each choice $\alpha_1, \ldots, \alpha_l\in \F_q$ there exists precisely one column vector $v=(\beta_1, \ldots, \beta_{b-1}, 0,\ldots,0)^t\in \F_q^n$ with $\beta_{k_\nu}=\alpha_\nu, \nu=1, \dots, l$ and one lidempotent $[A^v]\in \cO_A^r$ which coincides at all positions with $[A]$ besides those of the $b$-th column which is given by $v$. Moreover for $a<t<b$, with $ t\notin \{k_1, \ldots,k_l\}$ there exist unique elements $\lambda_1, \ldots, \lambda_l$ such that row $t$ of $\tilde A(a,b)$ is the linear combination of rows $k_1, \ldots, k_l$ of $\tilde A(a,b)$ with coefficients $\lambda_1, \ldots, \lambda_l$. Then the entry $\beta_t$ of column $b$ of $[A^v]$ at position $(t,b)$ is given as 
\begin{equation}\label{6.10}
\beta_t=A_{tb}+\sum_{\nu=1}^l \lambda_\nu (\beta_{k_\nu}-A_{k_\nu b}),
\end{equation}
where $A_{i b}$ is the entry at position $(i, b)$ in $[A]$. 
\end{Prop}\
\begin{proof}
This is an easy exercise in linear algebra using \ref{6.4}.
\end{proof}

We call the positions $\{(k_\nu, b)\,|\,1\leqslant \nu \leqslant l\}$ in the above proposition  in column $b$ of $A$ {\bf free positions} (of column $b$). One obvious choice for those consists of all  hook intersections on column $b$. Indeed, if $[A]$ is a normal template, for which all normal supplementary conditions are $0$, that is $\supp(A) = \main[A]$,  this is the only choice of free positions in column $b$. Our strategy will be, to replace some (indeed as many as possible) of the non normal hook intersections by positions $(i,b)$ such that row $i$ is normal containing non-zero supplementary but no main condition in $\Gamma_{ab}$.  In \ref{original notation for right} we have seen, that the lidempotents $[D]\in \cO_A^r$  may be described by choosing freely arbitrary entries $D_{ij}\in \F_q$ for  positions $(i,j)\in \Pl(\p, J)$, then all other entries are linearly dependent from the chosen ones. The previous Proposition \ref{6.9} may be considered as  a local version of this, namely that dependent positions in column $b$ of lidempotents in $\cO_A^r$ depend really only on the entries at free positions $(a,b)\in \Pl(\p,J)$ in column $b$, that is one can do the calculation of the dependent entries column wise.

\begin{Remark}\label{6.11}
We remark in passing, that \ref{6.9} has the following curious consequence: If  $\cC_i$ denotes for each $2\leqslant i \leqslant n$  the set of vectors occurring as $i$-th column vector in some $[D]\in \cO_A^r$, then for all choices $(c_2, \cdots, c_n)\in \cC_2\times \cdots\times  \cC_n$ the lidempotent $[C]$ is contained in $\cO_A^r$, where  $c_i$ is the $i$-th column vector of $C$. In particular
\begin{equation*}
q^{|\Pl(\p,J)|}=|\cO_A^r|=\prod_{i=2}^n |\cC_i|
\end{equation*} 
\end{Remark}

\begin{Defn}\label{6.12}
Keep the notation of \ref{6.8}. Let $B$ be the submatrix arising from $\tilde A(a,b)$ by deleting  rows $1, \ldots, a$ and all non normal rows. Let $k$ be the rank of $B$, then $k\leqslant l$ and if rows $n_1, \ldots, n_k$ in $B$ are linearly independent, they are linearly independent as well as rows of $\tilde A(a,b)$. Thus we may choose rows $n_1, \ldots, n_k$ of $B$ as $\F_q$-basis of the row space of $B$.  Recall from the Set up \ref{6.8} that main conditions $(i_1, j_1), \ldots, (i_r, j_r)$  are the normal ones, hence rows  $i_1,i_2,\cdots, i_r$ of $B$ are linearly independent, since each of them contains a main condition, which sits on different columns. Thus we choose
first rows  $i_1,i_2,\cdots, i_r$ and then $k-r$ many  rows from the remaining rows of $B$, namely rows $n_{r+1}, \cdots, n_k$, such that these together with the rows $i_1,i_2,\cdots, i_r$ are linearly independent rows in $\tilde B$. We obtain a $(k\times l)$-submatrix $\tilde B$ of $B$ such that the row space of $\tilde B$ is the same as the row space of $B$ and in particular $\rank(\tilde B)=\rank(B)=k $. Therefore we have in view of \ref{6.9}:
\end{Defn}

\begin{Cor}\label{6.13}
The normal positions $(i_1,b), \ldots, (i_r,b), (n_{r+1},b), \ldots,(n_k,b)$ of column $b$ of $[A]$ can be chosen to be free positions, and $k$ is the maximal number of free normal positions in column $b$. \hfill $\square$
\end{Cor}

Thus there remain $l-k$ many free positions on column $b$, which have to be non normal and therefore are non normal main hook intersections. How to choose those is explained now: Recall that $\tilde B$ has rank $k$ and its rows are labelled by $\{i_1, \ldots, i_r, n_{r+1}, \ldots, n_k\}$ (not necessarily in that order). In particular the normal main conditions $(i_1, j_1), \ldots,(i_r,j_r)$ are still positions of $\tilde B$. Since these are in different rows and above them are zero entries we conclude that columns $j_1,\ldots,j_r$ in $\tilde B$ are linearly independent. Now since we have chosen already all columns containing a normal main condition, the remaining columns of $\tilde A(a,b)$ contain a non normal main condition, and hence can be labelled by $j_{r+1}, \ldots, j_l$. We choose $k-r$ columns of these such that these together with columns $j_1, \ldots, j_r$ are linearly independent columns of $\tilde B$. If necessary by reordering we may assume, that we took columns $j_{r+1}, \ldots, j_k$, then the submatrix of $\tilde B$ consisting of columns $j_1, \ldots, j_r, j_{r+1}, \ldots, j_k$ has hence rank $k$ and is an invertible $k\times k$-matrix. 

\begin{Lemma}\label{6.14}
The rows with $s\in \{i_1, \ldots,i_r, n_{r+1},\ldots, n_k, i_{k+1}, \ldots,i_l\}=\cF$ of $\tilde A(a,b)$ are linearly independent. Thus $(h,b)$ with $h\in \cF$ can be chosen as free positions on column $b$ of $[A]$.
\end{Lemma}
\begin{proof}
We have to show that the matrix consisting of the rows $s$ with $s\in \cF$ has rank $l$. For this we can reorder rows and columns as follows:
\begin{equation}\label{figure}
\begin{picture}(200,270)
\put(0,0){\line(1,0){170}}
\put(0,0){\line(0,1){260}}
\put(170,260){\line(0,-1){260}}
\put(170,260){\line(-1,0){170}}

\put(0,200){\line(1,0){60}}
\put(60,200){\line(0,1){60}}
\put(5.5,245){\scriptsize normal main}
\put(10,235){\scriptsize conditions}
\put(10,225){\scriptsize and some}
\put(8.1,215){\scriptsize supplement}
\put(10,205){\scriptsize conditions}

\put(0,150){\line(1,0){170}}
\put(110,40){\line(0,1){220}}

\put(-10,253){\scriptsize $i_1$}
\put(-10,203){\scriptsize $i_r$}
\put(2,265){\scriptsize$j_{_1}$}
\put(52,265){\scriptsize$j_{_r}$}
\multiput(-7,248)(0,-5){8}{\circle*{1}}
\multiput(12,266)(5,0){8}{\circle*{1}}

\put(-18,193){\scriptsize $n_{_{r+1}}$}
\put(-10,153){\scriptsize $n_{k}$}
\put(62,265){\scriptsize$j_{_{r+1}}$}
\put(102,265){\scriptsize$j_{_{k}}$}
\multiput(-7,188)(0,-5){6}{\circle*{1}}
\multiput(82,266)(5,0){4}{\circle*{1}}

\put(112,265){\scriptsize$j_{_{k+1}}$}
\put(162,265){\scriptsize$j_{_{l}}$}
\multiput(132,266)(5,0){6}{\circle*{1}}

\put(0,90){\line(1,0){170}}

\put(115,135){\scriptsize the only non  }
\put(115,125){\scriptsize zero entries }
\put(115,115){\scriptsize are some}
\put(115,105){\scriptsize non normal}
\put(115,95){\scriptsize main cond.}

\put(-18,143){\scriptsize $i_{_{k+1}}$}
\put(-10,93){\scriptsize $i_{_{l}}$}
\multiput(-7,138)(0,-5){8}{\circle*{1}}

\put(-18,83){\scriptsize $i_{_{r+1}}$}
\put(-10,43){\scriptsize $i_{_{k}}$}
\put(0,40){\line(1,0){170}}
\put(60,40){\line(0,1){50}}
\multiput(-7,78)(0,-5){6}{\circle*{1}}

\put(62,82){\scriptsize the only non  }
\put(65,72){\scriptsize zero entries }
\put(65,62){\scriptsize are some}
\put(65,52){\scriptsize non normal}
\put(65,42){\scriptsize main cond.}

\put(15,23){\footnotesize $J$-normal rows linearly dependent }
\put(55,7){\footnotesize on rows of $\tilde B$}

\put(25,55){\Huge 0}
\put(130,55){\Huge 0}
\put(50,110){\Huge 0}
\put(130,195){\Huge $C$}

\put(190,260){\line(1,0){20}}
\put(190,90){\line(1,0){20}}
\put(200,185){\vector(0,1){75}}
\put(200,165){\vector(0,-1){75}}
\put(183,173){rows $\cF$}

\put(-60,260){\line(1,0){20}}
\put(-60,150){\line(1,0){20}}
\put(-50,225){\vector(0,1){35}}
\put(-50,190){\vector(0,-1){40}}
\put(-72,215){\footnotesize lin. indep. }
\put(-72,205){\footnotesize  normal}
\put(-72,195){\footnotesize   rows in $\tilde B$}
\end{picture}
\end{equation}
From this illustration one sees immediately that by some row operations one can delete the non zero entries of $C$ and hence the rows in $\cF$ are linearly independent. The rest follows from \ref{6.9}.
\end{proof}

The reader might wish to pause here and inspect the guiding example \ref{bigeg} at this point for column $8$ which contains the hook leg $h_{18}^l$. All hook intersections on column $8$ are non normal, hence $r=0$. The matrix $B$ consists of rows $5,6$ and $7$  to the right of column $8$, that is $B$ is the $3\times 3$-matrix
$$
B=\begin{pmatrix}
\alpha   & \delta     &\eta\\
\beta &  \epsilon  &  \mu\\
\gamma & \zeta       &  \nu
\end{pmatrix}.
$$ 

By assumption $B$ has rank $k=2$ and since $r=0$ we have as well $k-r=2$. Moreover the first two rows are assumed to be linearly independent, thus we have $n_{r+1},\ldots, n_k = n_1,n_2 = 5,6$ and hence free normal positions $(5,8), (6,8)$. Now there are $3$ non normal main hook intersection on column $8$, namely $(2,8),(3,8), (4,8)$ and hence $l=3$ and $l-k=1$. Thus we have to choose a further non normal hook intersection to obtain $i_{k+1}=i_3$.  Using \ref{lincomb} one sees easily that row $(z_2,0,0)$ is linearly independent from row $1$ and $2$ of $B$ (which are parts of rows $5$ and $6$ of $A$). Thus we can choose $i_3=2$ that is position $(2,8)$ as third free position.

So far we have fixed $l$ free positions $(h,b)$ with $h\in \cF$ on column $b$ of $[A]$ below the main condition $(a,b)$, (thus positions on $h_{ab}^l$). Formula \ref{6.10} describes how to calculate the other entries on $h_{ab}^l$ in terms of the entries on the free positions. These other positions on that hook leg which may have non zero values in lidempotents $[H]\in \cO_A^r$ are either $(i_\nu,b)$ with $r+1\leqslant \nu \leqslant k$ or of the form $(s,b)\in J$, where $s\notin \cF$ but row $s$ is normal in $[A]$.

Let $r+1\leqslant \nu \leqslant k$ and $[H]\in \hat V$. Suppose $H_{i_\nu,b}=\alpha\in \F_q\setminus\{ 0\}$, and  that $H$ and $A$ coincide in all other positions of $J$. Then we have:

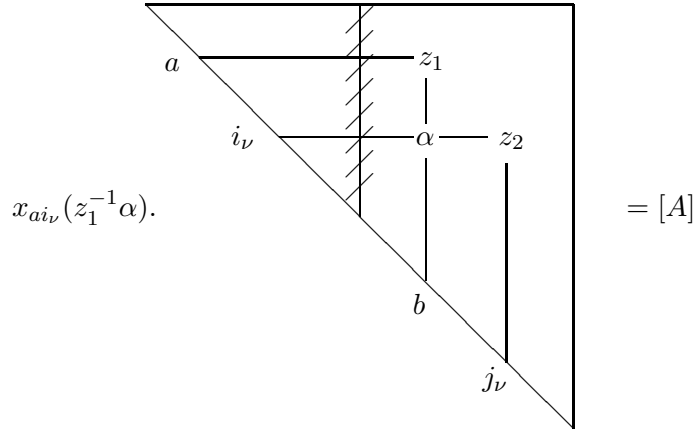
\begin{center}
\begin{picture}(220,160)
\put(-50,80){$x_{ai_\nu}(z_1^{-1}\alpha).$}
\put(180,80){$=[A]$}

\put(0,160){\line(1,0){160}}
\put(0,160){\line(1,-1){160}}
\put(160,160){\line(0,-1){160}}

\put(20,140){\line(1,0){80}}
\put(102,137){$z_1$}
\put(105,132){\line(0,-1){17}}
\put(101,107){$\alpha$}
\put(105,102){\line(0,-1){46}}
\put(7,135){$a$}
\put(100,43){$b$}


\put(50,110){\line(1,0){50}}
\put(110,110){\line(1,0){18}}
\put(132,107){$z_2$}
\put(135,100){\line(0,-1){75}}
\put(32,107){$i_\nu$}
\put(126,16){$j_\nu$}

\put(80,80){\line(0,1){80}}
\multiput(75,86)(0,9){8}{\line(1,1){10}}
\end{picture}
\end{center}

More generally one shows easily by the same argument:
\begin{Lemma}\label{6.16}
Let $[H]\in \hat V$  differing from a template $[A]\in \hat V$ only on non normal main hook intersections of $[A]$. Then there exists $u\in U$ such that $u.[H]=[A]$.
\end{Lemma}

\begin{Defn}\label{6.17}
For any finite set $\cS$ let $\F_q^\cS$ be $\F_q$-vector space with basis $\cS$ where the elements of $\F_q^\cS$ are written as column vectors indexed by $\cS$. 
We set $$\cS=\{n_{r+1}, \ldots, n_k\}, \,\cM=\{i_{k+1}, \ldots, i_l\}\text{ and } \cD=\{i_{r+1}, \ldots, i_k\},$$ thus $(i,b)$ is a supplementary normal condition and free position for $i\in \cS$, a free non normal hook intersection for $i\in \cM$ and for $i\in\cD$ a position of a non normal hook intersection whose value in $[A].g$, ($g\in U$ given in \ref{6.3}) depends on the values on position $(j,b)$ with $j\in \cF$ (see \ref{6.13} and \ref{6.9}). Moreover we define $\cH=\cM\cup \cD$, (that is the set of row indices of the non normal hook intersections), $\cF^\circ=\cS\cup \cM\subseteq \cF$, (that are the  row indices of the set of free positions not to the left of a normal main condition).
Note that all these sets are subsets of $\{a+1,\ldots,b-1\}\subseteq\{1,\ldots, b-1\}$.
\end{Defn}

\begin{Lemma}\label{6.18}
Let $v=(\beta_i)_{i\in \cF^\circ} \in \F_q^{\cF^\circ}$ and let $[A^v]\in \cO_A^r$ be the unique lidempotent which agrees with $[A]$ on all positions except those of hook leg $h_{ab}^l$, where those are given by
\[
A^v_{ib}=\begin{cases}
0  &  i\in \{i_1,\ldots,i_r\}\subseteq \cF, \text{ (rows of normal main hook intersections)}\\
\beta_i& i\in \cF^\circ\\
\text{given by \ref{6.9}} & i\notin \cF
\end{cases}
\]
Then 
there exists $y_v\in U_J$ such that $y_v.[A^v]\in \hat V$ is a normal template, denoted by $[A(v)]$. 
\end{Lemma}
\begin{proof}
Note that our definition here of $[A^v]$ differs slightly from that given in \ref{6.9} where we use a slightly different $v$.

$A^v$ differs from $A$ only on hook  leg $h_{ab}^l$, that is in column $b$. Moreover for all normal main hook intersections $(r,b)$ we have $A^v_{rb}=A_{rb}=0$ and $A^v_{ib}\neq 0$ only, if $(i,b)$ is a normal supplementary condition or a non normal main hook  intersection. By \ref{6.16} we find $y_v\in U$ such that $y_v.[A^v]$ differs from $[A^v]$ only at the non normal main hook intersections in column $b$, which are in $y_v.[A^v]$ zero. From this it follows immediately that $y_v.[A^v]$ is a normal template.
\end{proof}

\begin{Cor}\label{6.19}
Let $v,\, y_v\in U,\,[A^v]\in \cO_A^r, [A(v)] \in \hat V$ be as in \ref{6.18}. Then $[A^v]=y_v^{-1}.[A(v)]\in \cO_A^r\cap \cO_{A(v)}^l$ and the left action of $y_v^{-1}$ is an isomorphism from $\C\cO_A^r$ onto $\C \cO_{A(v)}^r$.\hfill $\square$
\end{Cor}

\begin{Lemma}\label{6.20}
Suppose $v\in \F_q^{\cF^\circ}$ with $v_i=A_{ib}$ for $i\in \cS$ and $v_i \in \F_q$ for $i\in \cM$. Let $[A^v]\in \cO_A^r, [A(v)] \in \hat V$ be defined as in \ref{6.18}. Then $[A(v)]=[A]$.
\end{Lemma}
\begin{proof}
Recall  that $B$ is the submatrix arising from $\tilde A(a,b)$ by deleting (zero) rows $1, \ldots, a$ and all non normal rows. Let $s$ be one row label of $B$. By construction $A^v_{sb}=A_{sb}=A(v)_{sb}=0$ if row $s$ contains a normal main condition or $s\in \cS$. Moreover $A_{tb}=A(v)_{tb}=0$ for all non normal rows $t$, since $[A]$ and $[A(v)]$ are normal templates. So in particular $A_{sb}=A(v)_{sb}$ for all $s\in \cF$.

Let $s$ be again a row index for $B$ and $s\notin \cF$. Thus row $s$ is normal and $(s,b)$ is a supplementary normal condition for $\p$. Note that $A^v$ and $A(v)$ differ only on non normal hook intersections in column $b$, hence $A_{sb}^v=A(v)_{sb}$. We may write rows of $B$ as linear combination of rows $i_1, \ldots, i_r, n_{r+1}, \ldots, n_k$, (see \ref{figure}) with coefficients $\lambda_1, \ldots, \lambda_k\in \F_q$ respectively. Then by \ref{6.10}
\[
A_{sb}^v=A_{sb}+\sum_{\nu=1}^k \lambda_{l_\nu}(v_{l_\nu}-A_{l_\nu b})\quad\text{
where }\quad l_\nu=\begin{cases}i_\nu & \nu=1,\ldots,r\\ n_\nu & \nu=r+1,\ldots,k.\end{cases}
\]
Note for $\nu=1,\ldots,r$ we have $v_{l_\nu}=A_{l_\nu b}=0$ and for $\nu=r+1,\ldots,k$, we have  $v_{l_\nu}=A_{l_\nu b}$ by construction and hence we have shown $A_{sb}^v=A_{sb}$.
\end{proof}

\begin{Defn}
Let $[A]\in \hat V$ be a normal template with $\main[A]=\p$. Let $\cY=\cY(\p,J)$. 
For each $L\in \F_q^{\cY}$, we set $[_{_L}\!A]$ be the lidempotent in $\hat V$ given by 
\[
(_{_L}\!A)_{ij}=
\begin{cases}
L_{ij}, & \text{ for } (i,j)\in \cY\\
A_{ij}, & \text{ otherwise.}
\end{cases}
\]
Thus $[_{_L}\!A]$ arises from $[A]$ by filling all $\cY$-conditions given by $L$ into the appropriate positions in $[A]$.
\end{Defn}

\begin{Theorem}
Let $J\subseteq \Phi^+$ be column closed. Let $U=U_J, V=V_J$ and $[A]\in \hat V$ be a normal template  with $\main[A]=\p=\{(i_1,j_1), \ldots, (i_m, j_m)\}$, where  $1\leqslant j_1<j_2<\cdots< j_m\leqslant n$.  Let $(a,b)\in \p$. Define $\cS(b)=\cS, \cM(b)=\cM$ and $\cF^\circ(b)=\cF^\circ$ as in \ref{6.17} and $y_v\in U$ for $v\in \F_q^{\cF^\circ}$ as in \ref{6.18}. Then
\begin{itemize}
\item[1)] For $1\leqslant s \leqslant m$, let $\underline{v}=(v^1,v^2,\ldots,v^m)\in \F_q^{\cF^{\cS(j_1)}}\times \F_q^{\cF^{\cS(j_2)}}\times \cdots \F_q^{\cF^{\cS(j_m)}}$ where $v^s\in\F_q^{\cF^{\cS(j_s)}}\leqslant\F_q^{\cF^\circ(j_s)}$. Define
\[
[A(\underline v)]= [A(v^1)(v^2)\cdots(v^m)],
\]
then $[A(\underline v)]$ is a normal template.
If $[C]\in \hat V$ with $\main[C]=\p$ is a template  then $\C\cO_C^r\cong \C\cO_A^r$ if and only if $[C]=[_{_L}\!A(\underline v)]$ for some $L\in \F_q^\cY, \,\cY=\cY(\p, J)$ and $\underline v\in\F_q^{\cF^{\cS(j_1)}}\times \F_q^{\cF^{\cS(j_2)}}\times \cdots \F_q^{\cF^{\cS(j_m)}}$. Thus $$\cO_A^{\bi}= \mathop{\dot\bigcup}_{L, \underline v} \cO^r_{_{_L}\!A(\underline v)}\quad,$$ where $L$ runs through $\F_q^\cY$ and $\underline v$ runs through $\F_q^{\cF^{\cS(j_1)}}\times \F_q^{\cF^{\cS(j_2)}}\times \cdots \F_q^{\cF^{\cS(j_m)}}$.

\item [2)]
For $1\leqslant s \leqslant m$ we set $$W_s=W_{j_s}=\{v\in \F_q^{\cF^\circ(j_s)}\,|\, v_i=A_{ij_s} \text{ for } i\in \cS(j_s), \,v_i\in \F_q \text{ for } i\in \cM(j_s)\}.$$ Then for $v\in W_s$ the lidempotent $[A(v)]$ in \ref{6.18} equals $[A]$ and hence $[A^v]=y_v^{-1}.[A(v)]=y_v^{-1}.[A]\in \cO^r_A\cap \cO_A^l.$ In fact for $v^s\in W_s$ with $s=1,\ldots, m$, we set 
\vskip-0.4cm
$$\underline{v}=(v^m, v^{m-1}, \ldots, v^1)\in W_{m}\times W_{m-1}\times \cdots \times W_1=W.$$
For $\underline{v}\in W$ set $y(\underline{v})=y^{-1}_{v^m} \cdots y^{-1}_{v^1}$. Then 
\[
\cO^r_A\cap \cO_A^l=\{y(\underline{v}).[A]\,|\,\underline{v}\in W\}.
\]
\end{itemize}
\end{Theorem}
\begin{proof}
\begin{itemize}
\item [1)] By \ref{yiso} we have $\C\cO_{{_L}\!A(v)}^r\cong \C\cO_{A(v)}^r$ and hence if $[C]=[{_L}\!A(v)]$ we have $\C \cO_C^r\cong \C\cO_{A(v)}^r\cong \C\cO_A^r$ by \ref{6.19}. Note that in this way we have $q^N$ many different templates $[C]$ such that $\C \cO_C^r\cong \C\cO_A^r$ where $N=|\cY|+|\underline{S}|$ with $|\underline{S}|= |\sum_{s=1}^m |\cS(j_s)|$. Recall from \ref{orbitsizelr} that $\C\cO_A^{\bi}=\bigoplus_{\kappa}\C \cO_A^r$, thus we obtain: \begin{equation}\label{ineq1}
q^N=q^{|\cY|+|\underline{S}|}\leqslant \kappa
\end{equation}

For the necessity we need to argue that the inequality above is actually an equality, that is there are no more orbit modules isomorphic to $\C\cO_A^r$.  We shall prove it together with a proof of part 2).

\item [2)] The order we have chosen working through the columns through left to right guarantees the matrices $\tilde A(a,b)$ which we use for  column $b$  are unchanged by previous moves, (compare \ref{6.11}).

By lemma \ref{6.20} and induction on $s$ we have $[A(v)]=[A]$ and hence one sees easily that $\{y(\underline{v}).[A]\,|\,\underline{v}\in W\}\in \cO^r_A\cap \cO_A^l$, where at each column $j_s$ we have $q^{|\cM(j_s)|}$ choices for $v\in W_s$,  (see \ref{6.17}).
Let $|\underline{\cM}|= |\sum_{s=1}^m |\cM(j_s)|$. Then we have 
\begin{equation}\label{ineq2}
q^{|\underline{\cM}|}\leqslant |\cO^r_A\cap \cO_A^l|
\end{equation}
By \ref{6.17} we have $|\cS(j_s)|=k-r, |\cM(j_s)|=l-k$ hence 
\[
|\cS(j_s)|+ |\cM(j_s)|=k-r+l-k=l-r,
\]
which is the number of non normal hook intersections on $^J\!h_{i_s j_s}^l$. Together with $\cY\cap ^J\!h_{i_s j_s}^l$ we obtain all non normal positions on $^J\!h_{i_s j_s}^l$. Summing over $s=1,\ldots,m$ we therefore obtain:
\begin{equation}\label{ineq3}
|\underline{\cS}|+ |\underline{\cM}|+|\cY|=|\Pl(\p,J)|
\end{equation}
\begin{eqnarray*}
|\cO_A^r|&=&q^{|\Pl(\p,J)|}=q^{|\underline{\cS}|+|\cY|}\cdot q^{|\underline{\cM}|} \quad\quad\text{     by (\ref{ineq3})}
\\
&\leqslant& \kappa\cdot q^{|\underline{\cM}|}\leqslant  \kappa\cdot |\cO^r_A\cap \cO_A^l|\quad\quad\text{     by (\ref{ineq1}) and (\ref{ineq2})}\\
&=&|\cO_A^r|  |\quad\quad\text{     by \ref{orbitsizelr}}
\end{eqnarray*}
Therefore we have \[\kappa=q^{|\underline{\cS}|+|\cY|}, \quad \quad q^{|\underline{\cM}|}=|\cO^r_A\cap \cO_A^l|.\]
\end{itemize}

\end{proof}

This theorem looks complicated, but it is not. The key point is, that we can work locally and  proceed step by step by inspecting the main hooks. Given a main hook, we first count the main hook intersections on its hook leg to determine the number of free positions on it. Say there are $l$ many of those. From $l$ we subtract the number $r$ of normal main hook intersections, since those are always free and these do not contribute to non-trivial endomorphisms of $\C\cO_A^r$ nor produce distinct orbits whose linear spans are isomorphic to $\C\cO_A^r$ . Now we inspect the matrix $B$ defined in \ref{6.12} of rank $k$ say. We choose $k-r$ linearly independent rows of matrix $B$ which are also linearly independent of the $r$ many rows of $B$ with normal main conditions. All of these are parts of rows of $A$. The intersection of these rows of $A$ with our hook leg produce further $k-r$ many free positions on our hook leg, and each of this free positions contributes a count of $q$ many different orbits in the biorbit $\cO_A^{\bi}$. Finally we choose $l-k$ many non normal hook intersections on our hook leg, such that the corresponding rows of matrix $\tilde{A}(a,b)$ of \ref{6.8} are still linearly independent from the $k$ rows already chosen. This gives further $l-k$ many free positions on the hook leg, and each of those determines $q$ many  elements in $\cO_A^l\cap\cO_A^r$ (for each choice of entry $\alpha \in \F_q$ on it) and gives therefore a count of $q$ for that endomorphism ring. Taking in account the $\cY$-conditions on the chosen main hook leg as well and summing the exponents of $q$ gives precisely $q^{m_A}$, where $m_A$ denotes the number of non normal hook intersections in the hook leg. Summing the $q$-exponents over all main hooks we see, that we have found all endomorphism of $\C\cO_A^r$ as well as the isomorphic orbit modules. 
Moreover, one may use this to construct a representative in each biorbit $\cO_A^{\bi}$ that is a ``normal form" for lidempotents under the bi-action of $U_J$. For instance one could choose representatives with ``minimal'' support, (i.e. having as many zeros as possible). This can be determined by setting  $\underline v=(v^1,v^2,\ldots,v^m)={\bf 0}\in \F_q^{\cF^{\cS(j_1)}}\times \F_q^{\cF^{\cS(j_2)}}\times \cdots \F_q^{\cF^{\cS(j_m)}}$ during the construction of the normal template $[A(\underline v)]$ in the theorem above.

\section*{Acknowledgements}
We would like to thank the referee for interesting comments and valuable suggestions.


\providecommand{\bysame}{\leavevmode ---\ }
\providecommand{\og}{``} \providecommand{\fg}{''}
\providecommand{\smfandname}{and}
\providecommand{\smfedsname}{\'eds.}
\providecommand{\smfedname}{\'ed.}
\providecommand{\smfmastersthesisname}{M\'emoire}
\providecommand{\smfphdthesisname}{Th\`ese}

\end{document}